\documentclass[12pt]{article}
\usepackage{amsthm,amsmath,amssymb,latexsym,amscd}
\newcommand{\g}{\frak{g}}
\newcommand{\h}{\frak{h}}

\newcommand{\gl}{\frak{g}\frak{l}}

\newcommand{\e}{\varepsilon}

\renewcommand{\v}{\vee}
\newcommand{\s}{\mathbf{s}}
\newcommand{\ti}{\tilde}

\newcommand{\F}{\mathcal{F}}
\newcommand{\D}{\mathcal{D}}
\newcommand{\E}{\mathcal{E}}
\newcommand{\I}{\mathcal{I}}

\newcommand{\Q}{\mathcal{Q}}
\renewcommand{\P}{\mathcal{P}}
\renewcommand{\S}{\mathcal{S}}
\newcommand{\T}{\mathcal{T}}
\renewcommand{\O}{\mathcal{O}}

\newcommand{\dgLie}{\mathrm{dgLie}}

\newcommand{\Hom}{\mathrm{Hom}}

\newcommand{\Coder}{\mathrm{Coder}}

\newcommand{\End}{\mathrm{End}}

\newcommand{\unshuff}{\mathrm{unshuff}}

\newcommand{\Com}{\mathcal{C}om}

\newcommand{\Lie}{\mathcal{L}ie}
\newcommand{\LL}{\mathcal{L}\mathcal{L}}
\newcommand{\Leib}{\mathcal{L}eib}

\newcommand{\Perm}{\mathcal{P}erm}
\newcommand{\Poiss}{\mathcal{P}oiss}
\newcommand{\Zinb}{\mathcal{Z}inb}

\newcommand{\p}{\prime}
\newcommand{\pa}{\partial}
\renewcommand{\c}{\circ}
\newcommand{\ot}{\otimes}
\newcommand{\ol}{\overline}

\newtheorem{definition}{Definition}[section]
\newtheorem{lemma}[definition]{Lemma}
\newtheorem{claim}[definition]{Claim}
\newtheorem{proposition}[definition]{Proposition}
\newtheorem{theorem}[definition]{Theorem}
\newtheorem{corollary}[definition]{Corollary}
\newtheorem{remark}[definition]{Remark}
\newtheorem{example}[definition]{Example}

\date{}

\begin{document}

\title
{
Distributive laws in derived bracket construction and
homotopy theory of derived bracket Leibniz algebras
}
\author{K. UCHINO}
\maketitle
\section{Introduction}

Leibniz algebras introduced by Jean-Louis Loday
are by definition vector spaces equipped with binary bracket products $[.,.]$
satisfying the Leibniz identity.
$$
[x_{1},[x_{2},x_{3}]]=[[x_{1},x_{2}],x_{3}]+[x_{2},[x_{1},x_{3}]].
$$
Leibniz algebras arise in many areas of mathematics,
in particular, in \textbf{derived bracket construction}
(cf. Y. Kosmann-Schwarzbach \cite{Kos1}.)
Let $(\g,(.,.),d)$ be a differential graded (dg) Lie algebra
with a Lie bracket $(.,.)$ of degree $0$ and a differential $d$ of degree $+1$.
One can define on $\g$ a new bracket product by
$$
[x_{1},x_{2}]:=(dx_{1},x_{2}),
$$
which is called a \textbf{binary} derived bracket or derived bracket for short.
The derived bracket is an odd Leibniz bracket, namely,
it satisfies an odd version of the Leibniz identity (see eq. (\ref{oddloday}) below.)
Derived brackets are used in (Poisson-)geometry and in theoretical physics
to formulate various bracket formalisms (see \cite{Kos2} for the details.)\\
\indent
The main aim of this note is to study an operadic and an algebraic homology
theory of derived bracket Leibniz algebras.
The Lie algebra with the derived bracket $(\g,(.,.),[.,.])$
satisfies the following two conditions.
\begin{eqnarray}
\label{IA}[x_{1},(x_{2},x_{3})]&=&([x_{1},x_{2}],x_{3})+(x_{2},[x_{1},x_{3}]),\\
\label{IB}[(x_{1},x_{2}),x_{3}]&=&([x_{1},x_{2}],x_{3})-([x_{2},x_{1}],x_{3}),
\end{eqnarray}
where $|x_{1}|=|x_{2}|=|x_{3}|=0$.
It will be proved that $\{(\ref{IA}),(\ref{IB})\}$ is
the \textbf{minimal condition} that the derived bracket satisfies (except the Leibniz identity.)
We here consider a new type of algebra;
\textbf{Lie-Leibniz algebras} are by definition {\em even} Lie algebras
equipped with {\em odd} Leibniz brackets satisfying (\ref{IA}) and (\ref{IB}).
We should remark that Leibniz brackets of Lie-Leibniz algebras are not necessarily derived brackets.
The notion of Lie-Leibniz algebra is considered to be
an abstraction of the derived bracket construction.\\
\indent
Our first task is to study the operad of Lie-Leibniz algebras, which is denoted by $\LL$.
It is not easy to study Lie-Leibniz algebras without help of algebraic operad theory
(cf. Loday-Vallette \cite{LV}.)
The operad $\LL$ is a binary quadratic operad generated by $\Lie(2)\oplus\s\Leib(2)$.
Here $\Lie$ is the operad of Lie algebras,
$\Leib$ is the one of Leibniz algebras
and $\s\Leib$ is the odd version of $\Leib$.
The quadratic relation of $\LL$ is the Jacobi identity,
the odd Leibniz identity and $\{(\ref{IA}), (\ref{IB})\}$ (see (\ref{nla})-(\ref{nld}) below.)
We will prove that $\{(\ref{IA}), (\ref{IB})\}$ is a \textbf{distributive law}
in the sense of M. Markl \cite{MM} and J. Beck \cite{Beck}.
This implies that the operad $\LL$ is Koszul.\\
\indent
The second task is to study the bar construction of Lie-Leibniz algebras.
Since $\LL$ is Koszul, the strong homotopy (sh) Lie-Leibniz operad $\LL_{\infty}$
is well-defined as the quasi-free resolution over $\LL$.
Therefore, the notion of $\LL_{\infty}$-algebra or sh Lie-Leibniz algebra
is defined as a representation of $\LL_{\infty}$.
We will study two interesting subclasses of sh Lie-Leibniz algebras.
One is \textbf{sh Leibniz algebras equipped with invariant 2-forms} (or Lie brackets)
and the other is a \textbf{derived homotopy construction of sh Lie algebra}.
The former naturally arises in higher geometry and
the latter provides a new type of derived bracket construction.
\medskip\\
\indent
This paper is organized as follows.\\
Section 2 is Preliminaries. We will recall odd Leibniz algebras,
classical derived bracket construction, algebraic operads, operadic version of
derived bracket construction and distributive law.\\
In Section 3.1, we will define the notion of Lie-Leibniz algebra and
observe some examples, i.e.,
invariant Lie algebras, Courant algebroids and omni-Lie algebras.
In 3.2, we will study the operad $\LL$ and prove that $\{(\ref{IA}),(\ref{IB})\}$
is a distributive law.
The distributive law implies that $\LL$ is decomposed into
$$
\LL\cong\Lie\odot\s\Leib,
$$
where $\odot$ is a certain monoidal structure (defined in Section 2.)
It is known that if two Koszul operads are unified via a distributive law,
then the induced operad is also Koszul.
In the case of $\LL$, $\Lie$ and $\s\Leib$ are both Koszul.
Hence $\LL$ is also so.\\
In Section 4, we will compute the Koszul dual operad of $\LL$.
The result of this section will be used in Sections 5 to construct
the bar complex of a Lie-Leibniz algebra.\\
In Section 5,
we will practice the bar construction of $\LL$-algebra
along the canonical method developed in Ginzburg-Kapranov \cite{GK}.\\
In Section 6, as an application of Section 5,
we will introduce the notion of strong homotopy Lie-Leibniz algebra
and prove some new results related with sh Lie and sh Leibniz algebras.
\medskip\\
\noindent\textbf{Acknowledgement}.
The author is grateful to Professor Jean-Louis Loday for his helpful comments and kind advice.

\section{Preliminaries}

Through the paper, all algebraic objects are assumed to be defined
over a fixed field $\mathbb{K}$ of characteristic zero
and graded linear algebra depends on Koszul sign convention.
Namely, the transposition of tensor product satisfies
$o_{1}\ot o_{2}\cong(-1)^{|o_{1}||o_{2}|}o_{2}\ot o_{1}$,
for any object $o_{i}$, where $|o_{i}|$ is the degree of $o_{i}$, $i\in\{1,2\}$.
We denote by $s$ the degree shifting operator of degree $+1$,
for instance, $|so|:=|o|+1$.
The degree of the inverse of $s$, $s^{-1}$, is $-1$.

\subsection{Classical derived bracket construction}
Leibniz algebras (or called Loday algebras) are by definition vector
spaces $\g$ equipped with binary bracket products $[.,.]$
satisfying the Leibniz identity,
$$
[x_{1},[x_{2},x_{3}]]=[[x_{1},x_{2}],x_{3}]+[x_{2},[x_{1},x_{3}]],
$$
where $x_{1},x_{2},x_{3}\in\g$.
If the bracket is anti-commutative, then it is a Lie bracket.
\begin{definition}[odd Leibniz algebras]
Let $(\g,[.,.])$ be a graded space with a binary
bracket of degree $+1$.
The pair $(\g,[.,.])$ is called a Leibniz algebra of degree $+1$,
if the identity below is satisfied,
\begin{eqnarray}\label{oddloday}
(-1)^{|x_{1}|}[x_{1},[x_{2},x_{3}]]+[[x_{1},x_{2}],x_{3}]+(-1)^{|x_{1}||x_{2}|+|x_{2}|}[x_{2},[x_{1},x_{3}]]=0,
\end{eqnarray}
which is the odd version of the Leibniz identity above.
\end{definition}
If $[x_{1},x_{2}]$ is an odd Leibniz bracket,
then the redefined bracket $[x_{1},x_{2}]^{\p}:=(-1)^{|x_{1}|}[x_{1},x_{2}]$
satisfies the classical formula,
\begin{equation}\label{evenloday}
[x_{1},[x_{2},x_{3}]^{\p}]^{\p}
=[[x_{1},x_{2}]^{\p},x_{3}]^{\p}+(-1)^{(|x_{1}|+1)(|x_{2}|+1)}[x_{2},[x_{1},x_{3}]^{\p}]^{\p}.
\end{equation}
If $[.,,]_{0}$ is an even Leibniz bracket defined on $\g$,
then $s[.,.]_{0}(s^{-1}\ot s^{-1}):s\g\ot s\g\to s\g$ is an odd Leibniz bracket
on the shifted space $s\g$.
Because $|s|:=+1$,
$$
[sx_{1},sx_{2}]:=
s[.,.]_{0}(s^{-1}\ot s^{-1})(sx_{1},sx_{2})=(-1)^{|x_{1}|+1}s[x_{1},x_{2}]_{0}.
$$
\begin{remark}
Odd Lie algebras are by definition vector spaces with commutative products $(.,.)$
of degree odd satisfying the odd Leibniz identity above.
The odd Lie bracket is graded commutative, that is,
$(x_{1},x_{2})=(-1)^{|x_{1}||x_{2}|}(x_{2},x_{1})$.
If $(.,.)_{0}$ is an even Lie bracket,
then $s^{-1}(.,.)_{0}(s\ot s)$ is an odd Lie bracket.
\end{remark}
\indent
Let $(\g,(\cdot,\cdot),d)$ be a dg Lie algebra
with a Lie bracket $(\cdot,\cdot)$ of degree $0$ and a differential $d$
of degree $+1$. Define a new bracket $[.,.]$ by
\begin{equation}\label{compd}
[x_{1},x_{2}]:=(dx_{1},x_{2}),
\end{equation}
where $x_{1},x_{2}\in\g$.
This bracket is called a \textbf{binary derived bracket} of the Lie bracket.
It is easy to check that the derived bracket satisfies the odd Leibniz identity.
\begin{remark}[original formula \cite{Kos1}]
If we put $[x_{1},x_{2}]^{\p}:=(-1)^{|x_{1}|}(dx_{1},x_{2})$,
then this bracket satisfies (\ref{evenloday}).
We will use (\ref{compd}) as the definition of derived bracket,
because it is compatible with an operadic derived bracket in Section 2.3.
\end{remark}
\begin{remark}[bracket degree \cite{Kos1}]\label{bracketdegree}
If the degree of the Lie bracket $(.,.)$ is $b$, then the derived bracket is defined as
$$
[x_{1},x_{2}]:=(-1)^{|b|}(dx_{1},x_{2}).
$$
In the case of the binary derived bracket, one can forget the sign $(-1)^{|b|}$, even if $b=odd$.
However, in the case of homotopy derived brackets introduced in Section 6.2,
the sign associated with the bracket degree is important and essential.
\end{remark}

\subsection{Algebraic operads (\cite{LV})}

Let $S_{n}$ be the $n$th symmetric group and let $\P(n)$ an $S_{n}$-module\footnote{
more correctly, the group ring $\mathbb{K}S_{n}$-module.
}.
A collection of $\P(n)$, $\P:=(\P(1),\P(2),...)$, is called an $\S$-module.
Given an $\S$-module $\P$,
one can define a functor (so-called Schur functor) as follows.
$$
\F_{\P}V:=\bigoplus_{n\ge 1}\P(n)\ot_{S_{n}}V^{\ot n},
$$
where $V$ is a $\mathbb{K}$-vector space
and $\ot_{S_{n}}:=\ot_{\mathbb{K}S_{n}}$.
In the category of $\S$-modules, a tensor product,
which is denoted by $\odot$, is defined by
\begin{definition}
$\P\odot\Q\iff\F_{\P\odot\Q}\cong\F_{\P}\F_{\Q}$.
\end{definition}
We consider a special $\S$-module, $\I:=(\mathbb{K},0,0,...)$.
It is easy to check that
$$
\P\odot\I\cong\P\cong\I\odot\P,
$$
namely, $\I$ is the unite element with respect to the tensor product.
\begin{definition}
Operads are by definition unital associative monoids
in the category of $\S$-modules.
Namely, an operad is an $\S$-module $\P$ equipped with
a binary product $\gamma:\P\odot\P\to\P$
and a unite map $\iota:\I\to\P$ satisfying unital associative law.
\end{definition}
The notion of operad morphism
is defined as a morphism of unital monoid in the category of $\S$-modules.\\
\indent
It is well-known that $\gamma:\P\odot\P\to\P$ is decomposed into partial products
$\gamma=(\c_{1},\c_{2},...)$, where $\c_{i}$ is a binary map such that
$$
\c_{i}:\P(m)\ot\P(n)\to\P(n+m-1),
$$
where $1\le i\le m$.
If $f\in\P(m)$, then $f$ is regarded as a formal multiplication of
\textbf{arity} $m$ and $f$ is expressed as $f=f(1,2,...,m)$.
The numbers, $1,2,...,m$, are called the \textbf{labels} of the \textbf{leaves} of $f$.
For any $g=g(1,...,n)\in\P(n)$, the operad structure $f\c_{i}g$
is defined as an insertion of $g$ in $f$ at the $i$th-leaf,
$$
(f\c_{i}g)(1,...,m+n-1):=f(1,...,i-1,g(i,...,i+n-1),...,m+n-1),
$$
for each $1\le i\le m$.
\begin{example}[End operad]\label{endoperad}
Let $V$ be a vector space.
Then the collection of $\End_{V}(n):=\Hom_{\mathbb{K}}(V^{\ot n},V)$
becomes an operad, which is called an endomorphism operad.
\end{example}
\begin{definition}[operad algebras]
Let $\P$ be an operad.
$\P$-algebras are by definition vector
spaces $V$ equipped with operad morphisms $\P\to\End_{V}$.
\end{definition}
The morphism $\P\to\End_{V}$ in above definition is called a representation of $\P$.
The notion of algebra (e.g. associative, Lie, Leibniz,...) is defined as a representation of operad.
\begin{definition}
The free operad over an $\S$-module
is the free associative monoid in the category of $\S$-modules
and the free operad over $\P$ is denoted by $\T\P$.
\end{definition}
\begin{definition}
Let $E=E(2)$ be a special $\S$-module such that $E(n\neq 2)=0$.
A \textbf{binary quadratic operad} over $E$ with $R$ is by definition
$$
\P:=\T{E}/(R),
$$
where $R$ is a sub $\S$-module of $(\T E)(3)$
and $(R)$ is the ideal generated by $R$.
By definition $\P(1):=\mathbb{K}$.
The generator of the ideal, $R$,
is called a \textbf{quadratic relation} of $\P$.
\end{definition}
Since $\T{E}$ is the free operad over $E$,
$\T{E}$ is generated by $E$ with $\gamma=(\c_{i})$.
For example,
$$
\T{E}(3)=<e\c_{1}e^{\p}, e\c_{2} e^{\p} \ | \ e,e^{\p}\in E>.
$$
If $\P$ is a binary quadratic (bq, for short) operad generated by $E$ with $R$,
then $\P(2)=E$ obviously.
Elements of $\P(2)$ are regarded as formal binary products,
which are denoted by $1*2$, $1\cdot 2$, $(1,2)$, $[1,2]$ and so on.
The structure of bq operad is completely determined by $E$ and $R$.
Hence we sometimes denote by $\P=(E,R)$ the bq operad.\\
\indent
We recall fundamental examples of binary quadratic operads.
\begin{example}[\cite{GK}]
The commutative associative operad is
$$
\Com:=\T(1\cdot 2)/(R_{\Com}),
$$
where $1\cdot 2=2\cdot 1$ is a formal commutative product
and $R_{\Com}$ is the space generated by the associative law,
$$
(1\cdot 2)\cdot 3-1\cdot(2\cdot 3)\equiv 0,
$$
where $(1\cdot 2)\cdot 3:=(1\cdot 2)\c_{1}(1\cdot 2)$
and $1\cdot(2\cdot 3):=(1\cdot 2)\c_{2}(1\cdot 2)$.
\end{example}
\begin{example}[\cite{GK}]
The Lie operad is
$$
\Lie:=\T\big((1,2)\big)/(R_{\Lie}),
$$
where $(1,2)=-(2,1)$ is a formal Lie bracket
and $R_{Lie}$ is the Jacobi identity,
$$
((1,2),3)+((3,1),2)+((2,3),1)\equiv 0.
$$
\end{example}
Roughly speaking, operads are abstract algebras without variables.\\
\indent
It is easy to see that for each $n$,
$$
\Com(n)=<1\cdot 2\cdot...\cdot n>\cong\mathbb{K}.
$$
If we put $1\ot 1:=1\cdot 2$, then
$\Com(n)={1}\ot\cdots\ot{1}={1}^{\ot n}$.
We will use this expression of $\Com$ in the next section.
\begin{lemma}\label{dimlieoperad}
$\dim\Lie(n)\cong(n-1)^{!}$ for each $n$.
\end{lemma}
\begin{proof}
An arbitrary Lie bracket is expressed as a linear combination of
the right normed brackets
$(\sigma_{1},...,(\sigma_{n-2},(\sigma_{n-1},n)))$,
where $\sigma\in S_{n-1}$.
\end{proof}
\begin{example}[\cite{Lod2}]
The Leibniz or Loday operad is
$$
\Leib:=\T([1,2],[2,1])/(R_{\Leib}),
$$
where $[1,2]$ is a formal Leibniz bracket,
$[2,1]$ is the transposition of $[1,2]$
and $R_{\Leib}$ is the even Leibniz identity,
$$
[1,[2,3]]-[[1,2],3]-[2,[1,3]]\equiv 0.
$$
\end{example}
Because $[\sigma_{1},...,[\sigma_{n-2},[\sigma_{n-1},\sigma_{n}]]]$
is the base of $\Leib(n)$,
\begin{lemma}[\cite{Lod2}]\label{dimleib}
$\dim\Leib(n)=n$ for each $n$.
\end{lemma}
\begin{example}[Zinbiel operad \cite{Zinb}]\label{operadZinb}
The Zinbiel operad is defined by
$$
\Zinb:=\T(1*2,2*1)/(R_{\Zinb}),
$$
where $1*2$ is a noncommutative product
and $2*1$ is the transposition of $1*2$.
The quadratic relation of $\Zinb$ is
$$
1*(2*3)-(1*2)*3-(2*1)*3\equiv 0.
$$
\end{example}
We will use this operad in Section 4.
\medskip\\
\indent
We recall differential graded operads.
\begin{definition}
A dg operad is an operad $\P$ such that
for each $n$, $(\P(n),d)$ is a complex and the differential $d$
is a derivation with respect to the operad structure.
\end{definition}
We denote by $\s$ (bold-faced) an operadic degree shifting operator.
If $\P=(\P(n))$ is an operad, $\s\P$ is the shift of $\P$ such that
$$
(\s\P)(n)\cong s^{-1}\ot\P(n)\ot s^{\ot n},
$$
where the degrees of $s$ and $s^{-1}$ (the inverse of $s$)
are respectively $+1$ and $-1$.
Hence $(\s\P)(n)$ is isomorphic to $s^{n-1}\P(n)\ot sgn_{n}$,
where $sgn_{n}$ is the sign representation of a symmetric group $S_{n}$.
The inverse of $\s$, $\s^{-1}\P$, is also defined by the same manner.

\subsection{Operadic derived brackets}

We recall the universal version of binary derived bracket construction
introduced in \cite{U2}.
\begin{definition}[Chapoton \cite{Chap}]
The permutation operad, $\Perm$, is a binary quadratic operad over
$(1\diamond 2,2\diamond 1)$ with the quadratic relation,
$$
(1\diamond 2)\diamond  3=(2\diamond  1)\diamond  3=1\diamond  (2\diamond  3).
$$
\end{definition}
We will observe that the operad $\Perm$ can be constructed
by using the commutative operad $\Com$ with a formal differential $d$.
Consider the free operad $\T(d,1\ot 1)$ over an $\S$-module $(d,1\ot 1)$,
where $d$ is a $1$-ary operator of degree $+1$
and $1\ot 1(:=1\cdot 2)$ is the binary commutative product.
\begin{definition}\label{operado}
$$
\O:=\T(d,{1}\ot{1})/(R_{\O}),
$$
where $R_{\O}$ is a quadratic relation generated by
\begin{eqnarray}
\nonumber dd&=&0,\\
\nonumber d({1}\ot{1})-d\ot{1}-{1}\ot{d}&=&0,\\
\label{assr} ({1}\ot{1})\ot{1}-{1}\ot({1}\ot{1})&=&0.
\end{eqnarray}
Here (\ref{assr}) is the associative law.
\end{definition}
\indent
The operad $\O$ naturally becomes a graded operad, $\O=(\O^{i})$,
whose degree is defined as the number of $d$.
Obviously, $\O^{0}\cong\Com$ and $\O^{n+1}(n)=0$ for each $n$.
\begin{lemma}
$(\O^{n-1}(n))\cong\s\Perm$.
\end{lemma}
\begin{proof}(sketch)
It is easy to prove that $(\O^{n-1}(n))$ is a suboperad of $\O$.
Up to degree, $1\diamond  2\cong d\ot 1$ and $2\diamond  1\cong 1\ot d$.
From the property of differential, $(\O^{n-1}(n))$ satisfies
$$
d\ot d\ot 1=-(d\ot 1)\c_{1}(d\ot 1)=(d\ot 1)\c_{1}(1\ot d)=(d\ot 1)\c_{2}(d\ot 1),
$$
which is the quadratic relation of $\s\Perm$.
\end{proof}
The binary bracket of $\Lie\ot\s\Perm:=(\Lie(n)\ot\Perm(n))$
is regarded as a derived bracket.
\begin{eqnarray*}
(d1,2)&\cong&(1,2)\ot(d\ot 1),\\
(1,d2)&\cong&(1,2)\ot(1\ot d).
\end{eqnarray*}
Hence $(1,2)\ot(d\ot 1)$ is an odd Leibniz bracket, in fact,
\begin{proposition}[\cite{Chap} see also \cite{V} or \cite{U2}]
$\s\Leib\cong\Lie\ot\s\Perm$.
\end{proposition}
The classical derived bracket construction is considered to be a representation
of this operadic identity.
In Section 3, we will introduce a generalization of above proposition.

\subsection{Distributive laws}

In this section, we recall the notion of operadic distributive law introduced in \cite{MM}.
Let $\P:=(E_{\P},R_{\P})$ and $\Q:=(E_{\Q},R_{\Q})$ be any bq operads.
We consider a linear map,
\begin{equation}\label{dmap}
\delta:\Q(2)\c\P(2)\Rightarrow\P(2)\c\Q(2),
\end{equation}
where
$\Q(2)\c\P(2):=<q\c_{i}p \ | \ q\in\Q(2),p\in\P(2),i\in\{1,2\}>$
and $\P(2)\c\Q(2)$ is also defined by the same manner.
Let $\Delta:=<x-\delta(x)>$ be the graph of the map
and let $\P\Q$ a new bq operad defined as follows.
$$
\P\Q:=(E_{\P}\oplus E_{\Q},R_{\P}\oplus\Delta\oplus R_{\Q}).
$$
We assume that the degree of $E_{\Q}$ is $+1$.
Then $\P\Q$ becomes a graded operad, in particular,
$$
\P\Q(4)=\P\Q^{0}(4)\oplus\P\Q^{1}(4)\oplus\P\Q^{2}(4)\oplus\P\Q^{3}(4).
$$
It is easy to see that
$\P\Q^{0}(4)\cong\P(4)$ and $\P\Q^{3}(4)\cong\Q(4)$.
The map $\delta$ or the relation $\Delta$
is called a \textbf{distributive law}, if the following identity naturally holds.
\begin{equation}\label{disdef4}
(\ol{\P}\odot\Q)(4)
\cong\P\Q^{1}(4)\oplus\P\Q^{2}(4),
\end{equation}
where $\ol{\P}:=(\P(2),\P(3))$.
There exists a natural epimorphism,
\begin{equation}\label{disdef5}
epi:(\ol{\P}\odot\Q)(4)\to\P\Q^{1}(4)\oplus\P\Q^{2}(4),
\end{equation}
which is induced from the universality of the tensor product $\odot$.
Hence (\ref{disdef4}) is equivalent to that (\ref{disdef5}) is mono.
If (\ref{disdef4}) holds, then $\P\Q$ is globally decomposed into $\P\odot\Q$.
\begin{example}
Let $\P=\Com$ be the commutative associative operad
and let $\Q=\Lie$ the Lie operad.
Then the derivation condition below is a distributive law.
\begin{eqnarray*}
\ \Lie(2)\c\Com(2)&\Rightarrow&\Com(2)\c\Lie(2),\\
\ [1,2\cdot 3]&\Rightarrow&[1,2]\cdot 3+2\cdot [1,3].
\end{eqnarray*}
The induced operad $\Com\odot\Lie(=:\Poiss)$ is the Poisson operad.
\end{example}

\section{Lie-Leibniz algebras}

\subsection{Definition and Examples}

\begin{definition}
Let $\g$ be a space with a Lie bracket $(.,.)$ of degree $0$
and a Leibniz bracket $[.,.]$ of degree $+1$.
We call the triple $(\g,(.,.),[.,.])$ a \textbf{Lie-Leibniz algebra} of bidegree $(0,1)$,
if the following two relations are satisfied,
\begin{eqnarray}
\label{LLre01}[x_{1},(x_{2},x_{3})]&=&([x_{1},x_{2}],x_{3})
+(-1)^{(|x_{1}|+1)|x_{2}|}(x_{2},[x_{1},x_{3}]),\\
\label{LLre02}[(x_{1},x_{2}),x_{3}]&=&([x_{1},x_{2}],x_{3})-(-1)^{|x_{1}||x_{2}|}([x_{2},x_{1}],x_{3}),
\end{eqnarray}
where $x_{1},x_{2},x_{3}\in\g$.
\end{definition}
In above definition,
if the degree of Lie bracket is $p$ and the one of Leibniz bracket is $q$,
then the bidegree is by definition $(p,q)$. However, $p$ and $q$ are not independent.
Namely, the parity of $p$ is different from the one of $q$.
In this section, we consider the case that $p$ is even and $q$ is odd.\\
\indent
The derived bracket $(dx_{1},x_{2})$ naturally satisfies
the defining condition of Lie-Leibniz algebras.
We should remark that
(\ref{LLre01}) is a consequence of the Jacobi identity,
however, (\ref{LLre02}) is not so.\\
\indent
We recall some examples of Lie-Leibniz algebras,
which arise in geometry.
\begin{example}[Lie algebra with invariant 2-form]\label{LAI2}
Let $(\h,(.,.),[.,,])$ be a $\mathbb{K}$-Lie algebra
with an invariant nondegenerate symmetric 2-form $(.,.):\h\ot\h\to\mathbb{K}$.
We consider a graded space $\g:=\h\oplus\mathbb{K}$,
where  by definition $|\h|:=1$ and $|\mathbb{K}|:=0$.
The two brackets $(.,.)$ and $[.,.]$ can be extended on $\g$ by
$[\h,\mathbb{K}]=[\mathbb{K},\mathbb{K}]=(\h,\mathbb{K})=(\mathbb{K},\mathbb{K})=0$.
Then $\g$ becomes a Lie-Leibniz algebra of bidegree $(-2,-1)$.
The identity (\ref{LLre01}) is equivalent to the following
invariance condition and (\ref{LLre02}) is trivial.
$$
[x_{1},(x_{2},x_{3})]=0=([x_{1},x_{2}],x_{3})+(x_{2},[x_{1},x_{3}]),
$$
where $x_{1},x_{2},x_{3}\in\h$.
\end{example}
\begin{remark}
If $\g$ is an even Leibniz algebra of homogeneous, then
it can be regarded as an odd algebra without suspension $s(-)$,
because the odd Leibniz identity (\ref{oddloday}) has the same form as the even version
when the parities of three variables are all odd.
\end{remark}
The sheaf version of above example is known as Courant algebroid.
\begin{example}[Courant algebroids \cite{Kos3}]
Let $E\to M$ be a vector bundle over a smooth manifold $M$
equipped with a smooth nondegenerate symmetric pairing $(.,.)$ on $E$
and a smooth $\mathbb{R}$-bilinear bracket $[.,.]$ on $\Gamma E$,
where $\Gamma E$ is the space of smooth sections of $E$.
The pairing $(.,.)$ is $C^{\infty}(M)$-bilinear on $\Gamma E$, i.e.,
$(.,.):\Gamma E\ot_{C^{\infty}(M)}\Gamma E\to C^{\infty}(M)$.
Suppose that there exists a bundle map, or derivation representation
$\rho:E\to TM$, where $TM$ is the tangent bundle on $M$.
We consider a graded space $\g:=\Gamma E\oplus C^{\infty}(M)$,
where by definition $|\Gamma E|:=1$ and $|C^{\infty}(M)|:=0$.
The bracket $[.,.]$ can be extended on $\g$ by the semi-direct product,
$$
[e_{1}\oplus f_{1},e_{2}\oplus f_{2}]_{\g}:=[e_{1},e_{2}]\oplus\rho(e_{1})(f_{2})+\rho(e_{2})(f_{1}),
$$
where $e_{i}\oplus f_{i}\in\g$, $i\in\{1,2\}$. The degree of $[.,.]_{\g}$ is $-1$ on $\g$.
The pairing is also extended on $\g$ as a Lie bracket of degree $-2$.
The quadruple $(E,(.,.),[.,.],\rho)$ is a Courant algebroid if and only if
$(\g,(.,.),[.,.]_{\g})$ is a smooth Lie-Leibniz algebra of bidegree $(-2,-1)$.
\end{example}
It is known that
the Lie bracket of Courant algebroid is a symplectic structure (or symplectic bracket)
and the Leibniz bracket of Courant algebroid (so-called Courant bracket)
is a derived bracket of the symplectic bracket (cf. Roytenberg \cite{Roy}.)\\
\indent
The next example, omni-Lie algebras, is regarded as a toy model of Courant algebroid.
\begin{example}[Omni-Lie algebras in Weinstein \cite{We}]
Let $V$ be a vector space.
The omni-Lie algebra over $V$ is the space $\E_{V}:=\gl(V)\oplus V$
equipped with a Leibniz bracket $[.,.]$ and a symmetric pairing $(.,.)$,
where the two brackets are defined by
\begin{eqnarray*}
[g_{1}+v_{1},g_{2}+v_{2}]&:=&[g_{1},g_{2}]+g_{1}(v_{2}),\\
(g_{1}+v_{1},g_{2}+v_{2})&:=&g_{1}(v_{2})+g_{2}(v_{1}),
\end{eqnarray*}
where $g_{1},g_{2}\in\gl(V)=\End(V)$ and $v_{1},v_{2}\in V$.
In a similar way as above examples, we put $|\E_{V}|:=1$ and $|V|:=0$
and consider a graded space $\g:=\E_{V}\oplus V$.
The structure of omni-Lie algebra can be extended on $\g$, and then,
the triple $(\g,(.,.),[.,.])$ becomes a Lie-Leibniz algebra of bidegree $(-2,-1)$.
\end{example}


\subsection{Lie-Leibniz operad}

We denote by $\LL$ the operad of Lie-Leibniz algebras.
This is a binary quadratic operad over
$$
\LL(2):=\Lie(2)\oplus\s\Leib(2)
$$
with the quadratic relation generated by
\begin{eqnarray}
\label{nla}[1,[2,3]]+[[1,2],3]+[2,[1,3]]&\equiv&0,\\
\label{nlb}[1,(2,3)]-([1,2],3)-(2,[1,3])&\equiv&0,\\
\label{nlc}[(1,2),3]-([1,2],3)+([2,1],3)&\equiv&0,\\
\label{nld}(1,(2,3))-((1,2),3)-(2,(1,3))&\equiv&0.
\end{eqnarray}
Here (\ref{nla}) is the odd Leibniz identity
and eq (\ref{oddloday}) is a representation of (\ref{nla}), that is,
$$
(\ref{oddloday})=\Big([1,[2,3]]+[[1,2],3]+[2,[1,3]]\Big)(x_{1}\ot x_{2}\ot x_{3}).
$$
We introduce a new operad, which is called a \textbf{deriving operad}.
\begin{definition}\label{derop}
The deriving operad, $\D$, is defined as a suboperad of $\O$ such that
$$
\D(n):=\O^{0}(n)\oplus\cdots\oplus\O^{n-1}(n),
$$
for each $n$.
The operad $\O$ has been defined in Definition \ref{operado}.
\end{definition}
In general $\D$ has the following form.
\begin{eqnarray*}
\D^{0}&=&\Com,\\
\D^{1}(2)&=&<d\ot{1} \ , \ {1}\ot d>,\\
\D^{1}(3)&=&<d\ot{1}\ot{1} \ , \ {1}\ot d\ot{1} \ , \ {1}\ot{1}\ot d>,\\
\D^{2}(3)&=&<d\ot d\ot {1} \ , \ d\ot {1}\ot d \ , \ {1}\ot d\ot d>,\\
\cdots&=&\cdots.
\end{eqnarray*}
Brackets in $\LL$ are correspond to derived brackets in $\Lie\ot\D$,
for example,
\begin{eqnarray*}
\ [1,[2,3]]&\cong&(1,(2,3))\ot(d\ot d\ot{1}),\\
\ [(1,2),3]&\cong&((1,2),3)\ot(d\ot {1}\ot{1}+{1}\ot d\ot {1}),\\
\ [1,(2,3)]&\cong&(1,(2,3))\ot(d\ot {1}\ot {1}),
\end{eqnarray*}
which implies
\begin{proposition}\label{keyprop}
$\LL\cong\Lie\ot\D$.
\end{proposition}
\begin{proof}
It is easy to prove that $\D$ is a bq operad generated by
$d\ot 1$, $1\ot d$ and $1\ot 1$.
The quadratic relation of $\D$ is naturally induced from
the associative law and the property of differential.\\
\indent
By using Proposition 15 in \cite{V} one can show that
\begin{equation}\label{LDLieD}
\Lie\ot\D\cong\Lie\c_{M}\D,
\end{equation}
where $\c_{M}$ is a white product of Manin (See Appendix.)
Hence $\Lie\ot\D$ is also a bq operad.
It is obvious that $\LL(2)\cong(\Lie\ot\D)(2)$.
The dimension of the quadratic relation $R_{\LL}$ is $13(=6+3+3+1)$,
which yields $\dim\LL(3)=14(=27-13)$.
Because $\dim\D(3)=7$, $\dim(\Lie\ot\D)(3)=14(=2\times 7)$.
The space $R_{\LL}$ is isomorphic to a subspace of the quadratic relation of $\Lie\ot\D$,
on the other hand, $\LL(3)\cong(\Lie\ot\D)(3)$.
Hence $R_{\LL}$ is identified with the quadratic relation of $\Lie\ot\D$.
\end{proof}
From $\dim\Lie(n)=(n-1)!$ (recall Lemma \ref{dimlieoperad} above),
\begin{corollary}
$$
\dim\LL(n)=(n-1)!\sum_{m=1}^{n}\binom{n}{m}.
$$
\end{corollary}
\indent
Now, we give the main result of this section.
The relation $\{(\ref{nlb}),(\ref{nlc})\}$ defines a mapping
\begin{equation}\label{maind}
\delta_{\LL}:\s\Leib(2)\c\Lie(2)\Rightarrow \Lie(2)\c\s\Leib(2).
\end{equation}
\begin{theorem}\label{maintheorem}
The map $\delta_{\LL}$ is a distributive law of $\Lie$ over $\s\Leib$, that is,
$$
(\ol{\Lie}\odot\s\Leib)(4)\cong\LL^{1}(4)\oplus\LL^{2}(4),
$$
where $\ol{\Lie}=(\Lie(2),\Lie(3))$.
\end{theorem}
We prove the theorem
by a direct computation of dimension.
\begin{proof}
From Proposition \ref{keyprop} we obtain
\begin{eqnarray*}
\dim\LL^{1}(4)&=&24,\\
\dim\LL^{2}(4)&=&36.
\end{eqnarray*}
To prove the theorem it suffices to show that
the dimension of $(\ol{\Lie}\odot\s\Leib)(4)$ is equal to $60(=24+36)$.\\
\indent
The subspace of $(\ol{\Lie}\odot\s\Leib)(4)$ of degree $+2$ is
\begin{equation}\label{ssdeg1}
\Lie(2)\c\Big(\s\Leib(2)\ot\s\Leib(2)\Big)\oplus\Lie(2)\c\s\Leib(3).
\end{equation}
The first term in (\ref{ssdeg1}) is generated by
the 12 monomials,
\begin{eqnarray*}
&&([1,2],[3,4]) \  \ ([1,3],[2,4]) \  \ ([1,4],[2,3])\\
&&([2,1],[3,4]) \  \ ([3,1],[2,4]) \  \ ([4,1],[2,3])\\
&&([1,2],[4,3]) \  \ ([1,3],[4,2]) \  \ ([1,4],[3,2])\\
&&([2,1],[4,3]) \  \ ([3,1],[4,2]) \  \ ([4,1],[3,2]).
\end{eqnarray*}
Hence we obtain
$$
\dim\Lie(2)\c\Big(\s\Leib(2)\ot\s\Leib(2)\Big)=12.
$$
The second term of (\ref{ssdeg1}) is generated by
the generators of 4-types,
\begin{equation*}
(1,\s\Leib(3)) \ , \ (2,\s\Leib(3)) \ , \ (3,\s\Leib(3)) \ , \ (4,\s\Leib(3)).
\end{equation*}
From Lemma \ref{dimleib}, $\dim\s\Leib(3)=6$, which yields,
$$
\dim\Lie(2)\c\s\Leib(3)=4\times 6=24.
$$
Thus we obtain
$$
\dim\Lie(2)\c\Big(\s\Leib(2)\ot\s\Leib(2)\Big)\oplus
\Lie(2)\c\s\Leib(3)=12+24=36.
$$
This number coincides with the dimension of $\LL^{2}(4)$.\\
\indent
We consider the subspace of $(\ol{\Lie}\odot\s\Leib)(4)$ of degree $+1$,
which has the form
\begin{equation}\label{ssdeg2}
\Lie(3)\c\s\Leib(2).
\end{equation}
Up to the Jacobi identity, (\ref{ssdeg2}) is generated by
\begin{eqnarray*}
(([1,2],3),4) && ((4,[1,2]),3)\\
(([2,1],3),4) && ((4,[2,1]),3)\\
(([1,3],2),4) && ((4,[1,3]),2)\\
(([3,1],2),4) && ((4,[3,1]),2)\\
(([1,4],2),3) && ((3,[1,4]),2)\\
(([4,1],2),3) && ((3,[4,1]),2)\\
\cdots && \cdots.
\end{eqnarray*}
There are totally 24 terms.
This is the dimension of $\LL^{1}(4)$.
\end{proof}
The distributive law implies that $\LL$ is decomposed into
$\LL\cong\Lie\odot\s\Leib$.
By the theorem proved in \cite{MM}, we obtain
\begin{corollary}
The operad $\LL$ is Koszul\footnote{
When the bar complex over an operad $\P$ is resolution,
the operad is said to be Koszul.
In general, an operad is {\em good}, if it is Koszul.
}.
\end{corollary}
\begin{corollary}\label{ffleib}
Let $\g$ be an even Leibniz algebra and let $s^{-1}\g$ the shifted one.
The free Lie algebra, $\F_{\Lie}(s^{-1}\g)$,
becomes the free Lie-Leibniz algebra of bidegree $(0,1)$
in the category of odd Leibniz algebras.
In particular, if $\g$ is free, then the Lie-Leibniz algebra is also so.
\end{corollary}
\section{Koszul duals of $\LL$ and $\s\LL$}

\subsection{$\LL^{!}$}

The Koszul dual operad of $\P=(E,R)$ is
by definition $\P^{!}:=(E^{\v},R^{\bot})$,
where $E^{\v}$ is the invariant\footnote{
The pairing of the duality
satisfies $<e(12),e^{\v}(12)>=-<e(21),e^{\v}(21)>$,
where $e\in E$ and $e^{\v}\in E^{\v}$.
}  dual space of $E$ and $R^{\bot}$
is the orthogonal space of $R$.
A relation of dimension, $\dim\P(3)=\dim R^{\bot}$, holds.\\
\indent
It is known that the dual of a distributive law,
i.e., the dual map of $\delta$ defined in (\ref{dmap})
is also a distributive law.
Therefore, the Koszul dual of $\LL=\Lie\odot\s\Leib$ is
$$
\LL^{!}=(\s^{-1}\Zinb)\odot\Com.
$$
where $\Com=\Lie^{!}$ and $\Zinb=\Leib^{!}$ (recall Example \ref{operadZinb}.)\\
\indent
The aim of this section is to determine the quadratic relation of $(\s^{-1}\Zinb)\odot\Com$.
The dual of the Lie bracket $(1,2)$ is the commutative associative product $1\cdot 2$.
We denote the generators of $\s^{-1}\Zinb$ by $1*2$ and $2*1$,
which are the duals of the odd Leibniz brackets $[1,2]$ and $[2,1]$.
\begin{proposition}
The quadratic relation of $(\s^{-1}\Zinb)\odot\Com$ is generated by
\begin{eqnarray}
\label{dual1}1*(2*3)+(1*2)*3-(2*1)*3&\equiv&0,\\
\label{dual2}(1*2)\cdot 3-1*(2\cdot 3)+(1\cdot 2)*3&\equiv&0,\\
\label{dual3}1\cdot(2\cdot 3)-(1\cdot 2)\cdot 3&\equiv&0,
\end{eqnarray}
where
(\ref{dual1}) is the quadratic relation of $\s^{-1}Zinb$ and
(\ref{dual2}) is a distributive law of $\s^{-1}\Zinb$ over $\Com$.
\end{proposition}
\begin{proof}(Sketch)
Let $R_{ZC}$ be the quadratic relation generated by
(\ref{dual1}), (\ref{dual2}) and (\ref{dual3}).
The relation (\ref{dual1}) generates 6-basis,
(\ref{dual3}) generates 2-basis
and (\ref{dual2}) generates the following 6-basis,
\begin{eqnarray}
\label{ddrel01}
(1*2)\cdot 3-1*(2\cdot 3)+(1\cdot 2)*3&\equiv&0,\\
\nonumber (2*1)\cdot 3-2*(1\cdot 3)+(2\cdot 1)*3&\equiv&0,\\
\nonumber (2*3)\cdot 1-2*(3\cdot 1)+(2\cdot 3)*1&\equiv&0,\\
\nonumber (3*2)\cdot 1-3*(2\cdot 1)+(3\cdot 2)*1&\equiv&0,\\
\nonumber (3*1)\cdot 2-3*(1\cdot 2)+(3\cdot 1)*2&\equiv&0,\\
\nonumber (1*3)\cdot 2-1*(3\cdot 2)+(1\cdot 3)*2&\equiv&0.
\end{eqnarray}
Hence $\dim R_{ZC}=14$, which satisfies the consistency condition $\dim\LL(3)=\dim R_{ZC}$.
The pairing $<,>$ which defines the Koszul duality is defined by
$$
<\mu\c_{i}\nu\ot l \ , \ \mu^{\p}\c_{j}\nu^{\p}\ot l^{\p}>:=(-1)^{i}(-1)^{l}
\delta_{ij}\delta_{ll^{\p}}<\mu,\mu^{\p}><\nu,\nu^{\p}>,
$$
where $l$, $l^{\p}$ are labels of trees
and $\c_{i}$ is the operad structure.
For example,
\begin{eqnarray*}
<(\ref{ddrel01}) ,(\ref{nlc})>&=&<(1*2)\cdot 3-1*(2\cdot 3)+(1\cdot 2)*3 \ , \ [(1,2),3]-([1,2],3)+([2,1],3)>\\
&=&-<(1*2)\cdot 3 \ , \ ([1,2],3)>+<(1\cdot 2)*3 \ , \ [(1,2),3]>\\
&=&1-1=0.
\end{eqnarray*}
In this way one can show that $R_{ZC}=R^{\bot}_{\LL}$.
\end{proof}
\subsection{$\s\LL^{!}$}
We consider the Lie-Leibniz operad of bidegree $(-1,0)$, i.e., $\s^{-1}\LL$.
In this case, since the Lie bracket is odd,
the differential $d$ satisfies
$$
[d,(1,2)]=d(1,2)+(d1,2)+(1,d2)=0,
$$
where $[d,(1,2)]$ is a graded commutator.
Hence (\ref{nlc}) is modified to
\begin{equation}\label{oodnlc}
[(1,2),3]+([1,2],3)+([2,1],3)=0,
\end{equation}
on the other hand, (\ref{nlb}) does not change.
\begin{corollary}
The Koszul dual of $\s^{-1}\LL$ is $\Zinb\odot\s\Com\cong\s(\s^{-1}\Zinb\odot\Com)$.
The quadratic relation of $\Zinb\odot\s\Com$ is
\begin{eqnarray}
\label{odddual1}1*(2*3)-(1*2)*3-(2*1)*3&\equiv&0,\\
\label{odddual2}(1*2)\cdot 3-1*(2\cdot 3)-(1\cdot 2)*3&\equiv&0,\\
\label{odddual3}(1\cdot 2)\cdot 3+1\cdot(2\cdot 3)&\equiv&0.
\end{eqnarray}
Here $1*2,2*1\in\Zinb(2)$ are even and $1\cdot 2\in\s\Com(2)$ is odd.
\end{corollary}
Since $1\cdot 2\in\s\Com(2)$ is odd, it is anti-commutative, i.e., $1\cdot 2=-2\cdot 1$.
\begin{remark}
Odd commutative algebras are by definition vector spaces
equipped with anti-commutative products $\cdot$ of degree odd
satisfying the odd associative law,
\begin{eqnarray*}
x_{1}\cdot x_{2}&=&-(-1)^{|x_{1}||x_{2}|}x_{2}\cdot x_{1},\\
(x_{1}\cdot x_{2})\cdot x_{3}&=&-(-1)^{|x_{1}|}x_{1}\cdot(x_{2}\cdot x_{3}).
\end{eqnarray*}
Redefine a new product by $x_{1}\wedge x_{2}:=(-1)^{|x_{1}|}x_{1}\cdot x_{2}$.
Then it satisfies the classical identities,
$x_{1}\wedge x_{2}=(-1)^{(|x_{1}|+1)(|x_{2}|+1)}x_{2}\wedge x_{1}$
and
$(x_{1}\wedge x_{2})\wedge x_{3}=x_{1}\wedge(x_{2}\wedge x_{3})$.
\end{remark}
In the final of this section, we introduce two lemmas below,
which will be used in the next section.
\begin{lemma}\label{newlemma1}
We put $1*\cdots*m:=(((1*2)*3)*\cdots)*m$.
The operad $\Zinb\odot\s\Com$ satisfies
\begin{multline*}
(1*\cdots*i)\cdot(i+1*\cdots*n)=\\
\sum_{k=1}^{n-1}\Big((1*\cdots*i)^{k}\doublecap(i+1*\cdots*n-1)\Big)*n
-\sum_{k=i+1}^{n-1}\Big((1*\cdots*i-1)\doublecap(i+1*\cdots*n)^{i+k}\Big)*i\\
+\Big((1*\cdots*i-1)\doublecap(i+1*\cdots*n-1)\Big)*(i\cdot n),
\end{multline*}
where
$(1*\cdots*i)^{k}:=1*\cdots*(k\cdot k+1)*\cdots*i$
and $\doublecap$ is a shuffle product.
\end{lemma}
\begin{proof}(Sketch)
When $(i,n)=(2,1)$ or $(i,n)=(1,2)$,
the identity of the lemma is just (\ref{odddual2}).
Let $I=I(1,...,n)$ be the identity of the lemma,
for instance, $(\ref{odddual2})=I(1,2,3)$ holds.
By the assumption of induction, $I(1,...,n)$ holds.
This yields $I(1*2,3,...,n+1)$.
From the Zinbiel identity (\ref{odddual1}), we obtain
$$
I(1*2,3,...,n+1)=I(1,2,...,n+1).
$$
\end{proof}
\begin{corollary}\label{keycoro}
In above lemma, if $i=n-1$,
$$
(1*\cdots*n-1)\cdot n=\sum_{k=1}^{n-1}
1*\cdots*(k\cdot k+1)*\cdots*n.
$$
\end{corollary}
Suppose that $(A,*,\cdot)$ is a $\Zinb\odot\s\Com$-algebra.
From (\ref{odddual2}) and (\ref{odddual3}),
\begin{eqnarray*}
(a_{1}*a_{2})a_{3}&=&(-1)^{|a_{1}|}a_{1}*(a_{2}a_{3})+(a_{1}a_{2})*a_{3},\\
(a_{1}a_{2})a_{3}&=&-(-1)^{|a_{1}|}a_{1}(a_{2}a_{3}),
\end{eqnarray*}
where $a_{1},a_{2},a_{3}\in A$ and $a_{1}a_{2}:=a_{1}\cdot a_{2}$.
Corollary \ref{keycoro} is represented as follows.
$$
(a_{1}*\cdots*a_{n-1})a_{n}=
\sum_{k=1}^{m-1}(-1)^{|a_{1}|+\cdots+|a_{k-1}|}a_{1}*\cdots*a_{k-1}*(a_{k}a_{k+1})*\cdots*a_{n}.
$$
\begin{lemma}\label{zcder}
Let $d:A\to A*\cdots*A=A^{*m}$ be a derivation on
the $\Zinb\odot\s\Com$-algebra.
We put $da_{i}:=\sum_{(i_{1},...,i_{m})}a_{i_{1}}*\cdots*a_{i_{m}}$.
Then
\begin{multline*}
d\big((((a_{1}a_{2})a_{3})...)a_{n}\big)=\sum_{i=1}^{n}(-1)^{|a_{i}|(|a_{1}|+\cdots+|a_{i-1}|+m)+(m-i)(1+|d|)}\\
\sum_{(i_{1},...,i_{m})}\sum_{k=1}^{m}(-1)^{|a_{i_{1}}|+\cdots+|a_{i_{k-1}}|}
a_{i_{1}}*\cdots*(a_{i_{k}}a_{i_{k+1}})*\cdots*a_{im}*(((((a_{1}a_{2})a_{3})...)...\hat{a}_{i})...a_{n}),
\end{multline*}
where $\hat{\cdot}$ is an eliminator.
\end{lemma}
\begin{proof}(Sketch)
The identity of the lemma is followed from the commutative associative law
and Lemma \ref{newlemma1}. For each $i$,
\begin{eqnarray*}
\pm(((a_{1}a_{2})a_{3})...da_{i}...)a_{n}
&=&\pm da_{i}(((a_{1}a_{2})a_{3})...\hat{da}_{i}...)a_{n}\\
&=&\pm\sum_{(i_{1},...,i_{m})}\sum_{k=1}^{m}\pm\\
&&a_{i_{1}}*\cdots*(a_{i_{k}}a_{i_{k+1}})*\cdots*a_{im}*(((a_{1}a_{2})a_{3})...\hat{da}_{i}...)a_{n}.
\end{eqnarray*}
It is not necessarily to determine the sign of the identity.
\end{proof}

\section{Bar construction}

The aim of this section is to construct the bar complex
of Lie-Leibniz algebra along the canonical
method introduced in \cite{GK}.
\subsection{Preliminaries}
In this subsection we recall fundamental coalgebras, i.e.,
commutative coalgebras and Zinbiel coalgebras.\\
\indent
Let $V$ be a vector space and let $\bar{S}V$
the symmetric space over $V$ without unite,
$$
\bar{S}V:=V\oplus VV\oplus VVV\oplus\cdots,
$$
where $VV:=V\ot V/(v_{1}\ot v_{2}-(-1)^{|v_{1}||v_{2}|}v_{2}v_{1})$, $v_{1},v_{2}\in V$.
It is well-known that the cofree\footnote{in the category of nilpotent coalgebras}
commutative coalgebra over $V$ is $\bar{S}V$.
The commutative coproduct is defined as,
\begin{eqnarray*}
\Delta_{c}V&:=&0,\\
\Delta_{c}(v_{1}\cdots v_{n})&:=&
\sum_{\sigma\atop 1\le{i}\le n-1}\e(\sigma)(v_{\sigma(1)}\cdots v_{\sigma(i)} \ , \
v_{\sigma(i+1)}\cdots v_{\sigma(n)}),
\end{eqnarray*}
where $\e(\sigma)$ is the Koszul sign
and $\sigma$ is an $(i,n-i)$-unshuffle permutation\footnote{When $\sigma^{-1}$ is a shuffle permutation,
$\sigma$ is called an unshuffle permutation.}, that is,
$\sigma(1)<\cdots<\sigma(i)$ and $\sigma(i+1)<\cdots<\sigma(n)$.
Coderivations on $\bar{S}V$ are by definition endomorphisms $\pa:\bar{S}V\to\bar{S}V$ satisfying
\begin{equation}\label{defcoder}
(\pa\ot 1+1\ot \pa)\Delta_{c}=\Delta_{c}\pa.
\end{equation}
It is well-known that in general the space of coderivations
on the cofree operad-coalgebra over a base space
is isomorphic to the space of homomorphisms
from the coalgebra to the base space.
Hence in this case,
\begin{equation}\label{cohom1}
\Coder(\bar{S}V)\cong\Hom(\bar{S}V,V),
\end{equation}
where $\Coder(\bar{S}V)$ is the space of coderivations on $\bar{S}V$.
If $f:\bar{S}^{m}V\to V$, then the coderivation $\pa_{f}\cong f$
induced from $f$ is given by
\begin{multline}\label{cohom2}
\partial_{f}(v_{1}\cdots v_{n})=
\sum_{\sigma\atop m\le k\le n}\e(\sigma)(-1)^{|f|(|v_{\sigma(1)}|+\cdots+|v_{\sigma(i-1)}|)}\\
v_{\sigma(1)}\cdots v_{\sigma(i-1)}\cdot f(\overbrace{v_{\sigma(i)},...,v_{\sigma(k-1)},v_{k}}^{m})\cdot v_{k+1}\cdots v_{n},
\end{multline}
where $\sigma$ is $(i-1,m-1)$-unshuffle.
For example, when $m=2$ and $n=3$,
$$
\partial_{f}(v_{1}v_{2}v_{3})=
f(v_{1},v_{2})v_{3}+(-1)^{|v_{1}||f|}v_{1}f(v_{2},v_{3})+(-1)^{|v_{2}||f|+|v_{1}||v_{2}|}v_{2}f(v_{1},v_{3}).
$$
\begin{lemma}
The space $\g(=s^{-1}V)$ is a Lie algebra of degree $0$
if and only if $\bar{S}s\g(=\bar{S}V)$ is a dg coalgebra
with a binary codifferential $\pa$ of degree $-1$,
i.e., $\pa$ is binary and $\pa\pa=0$.
\end{lemma}
\begin{proof}
Let $(.,.)$ be a Lie bracket on $\g$.
Then $s(s^{-1},s^{-1})$ is an odd Lie bracket on $s\g$, which defines
a binary coderivation $\pa:=s(s^{-1},s^{-1})$ of degree $-1$.
The Jacobi identity is equivalent to $\pa\pa=0$.
\end{proof}
\indent
We recall the cofree Zinbiel coalgebra.
It is known that the tensor space $\bar{T}sV:=sV\oplus (sV)^{\ot 2}\oplus\cdots$
becomes the cofree Zinbiel coalgebra in the category of nilpotent coalgebras
(Ammar and Poncin \cite{AP}, see also \cite{U1}.)
The Zinbiel coproduct, $\Delta_{z}$, is defined by $\Delta_{z}sV:=0$ and
\begin{equation}\label{zinbcoproduct}
\Delta_{z}(sv_{1}\ot\cdots\ot sv_{n}):=\sum_{\sigma\atop 1\le i\le n-1}\e(\sigma)
(sv_{\sigma(1)}\ot\cdots\ot sv_{\sigma(i)},
sv_{\sigma(i+1)}\ot\cdots\ot sv_{\sigma(n-1)}\ot sv_{n}),
\end{equation}
where $\sigma$ is an $(i,n-1-i)$-unshuffle permutation.
The coproduct $\Delta_{z}$ satisfies the co-Zinbiel relation,
which is the dual of (\ref{odddual1}),
$$
(1\ot\Delta_{z})\Delta_{z}=(\Delta_{z}\ot 1)\Delta_{z}+(\tau\Delta_{z}\ot 1)\Delta_{z},
$$
where $\tau:\bar{T}sV\ot\bar{T}sV\to\bar{T}sV\ot\bar{T}sV$,
$\tau:c_{1}\ot c_{2}\mapsto (-1)^{|c_{1}||c_{2}|}c_{2}\ot c_{1}$
 is a transposition.
Coderivations on the Zinbiel coalgebra are defined by the same manner as (\ref{defcoder}).
The following identity also holds.
$$
\Coder(\bar{T}sV)\cong\Hom(\bar{T}sV,sV).
$$
Given $f:(sV)^{\ot m}\to sV$, the coderivation $\pa_{f}$ induced from $f$
has the same form as (\ref{cohom2}), that is,
\begin{multline}\label{cohom3}
\pa_{f}(sv_{1}\ot\cdots\ot sv_{n})=
\sum_{\sigma\atop m\le k\le n}\e(\sigma)(-1)^{|f|(|sv_{\sigma(1)}|+\cdots+|sv_{\sigma(i-1)}|)}\\
sv_{\sigma(1)}\ot\cdots\ot sv_{\sigma(i-1)}\ot
f(\overbrace{sv_{\sigma(i)},...,sv_{\sigma(k-1)},sv_{k}}^{m})\ot sv_{k+1}\ot\cdots\ot sv_{n},
\end{multline}
where $\sigma$ is the unshuffle permutation.
\begin{lemma}\label{zbar}
The space $\g(=s^{-1}V)$ is a Leibniz algebra of degree $1$
if and only if $\bar{T}ss\g(=\bar{T}sV)$ is a dg coalgebra with a binary codifferential of degree $-1$.
\end{lemma}
\begin{proof}
If $\g$ has a Leibniz bracket $[.,.]$ of degree $1$,
then $ss\g$ has a Leibniz bracket of degree $-1$,
which defines a coderivation,
$$
\pa_{Leib}:=ss[.,.](s^{-1}\ot s^{-1})(s^{-1}\ot s^{-1}).
$$
The Leibniz identity is equivalent to $\pa_{Leib}\pa_{Leib}=0$.
\end{proof}

\subsection{Cofree $\Zinb\odot\s\Com$-coalgebra}

In first we recall odd version of commutative coalgebras.
\begin{definition}\label{oddccalg}
Commutative coalgebras of degree $+1$
are by definition spaces $C$ equipped with
coproducts $\Delta_{c}$ of degree $+1$ satisfying
\begin{eqnarray*}
\tau\Delta_{c}&=&-\Delta_{c},\\
(\Delta_{c}\ot 1)\Delta_{c}&=&-(1\ot\Delta_{c})\Delta_{c},
\end{eqnarray*}
where $\tau:C\ot C\to C\ot C$,
$c_{1}\ot c_{2}\mapsto(-1)^{|c_{1}||c_{2}|}c_{2}\ot c_{1}$,
is the transposition.
\end{definition}
If $(C,\Delta_{c}^{0})$ is an even commutative coalgebra,
then $sC$ becomes an odd coalgebra, whose coproduct is
defined as
$$
\Delta_{c}:=(s\ot s)\Delta_{c}^{0}s^{-1}:sC\to sC\ot sC.
$$
\begin{definition}
$\Zinb\odot\s\Com$-coalgebras of bidegree $(0,1)$
are by definition spaces, $C$, equipped with
Zinbiel-coproducts $\Delta_{z}$ of degree $0$
and commutative coproducts $\Delta_{c}$ of degree $+1$
satisfying
$$
(\Delta_{z}\ot 1)\Delta_{c}=(1\ot \Delta_{c})\Delta_{z}+(\Delta_{c}\ot 1)\Delta_{z},
$$
which is the dual of (\ref{odddual2}).
\end{definition}
\begin{definition}
Coderivations on $\Zinb\odot\s\Com$-coalgebras
$(C,\Delta_{z},\Delta_{c})$ are
by definition endomorphisms $\partial:C\to C$ satisfying
\begin{eqnarray}
\label{DDZ}(\partial\ot 1+1\ot \partial)\Delta_{z}&=&\Delta_{z}\partial,\\
\label{DDC}(-1)^{|\partial|}(\partial\ot 1+1\ot \partial)\Delta_{c}&=&\Delta_{c}\partial.
\end{eqnarray}
Here (\ref{DDC}) is equivalent with $[\pa,\Delta_{c}]=0$.
\end{definition}
We should construct the cofree $\Zinb\odot\s\Com$-coalgebra,
before that, we set some convenient symbols.
\begin{definition}
For a given word $w_{1}...w_{n}$, we put
\begin{eqnarray*}
\unshuff_{i}(w_{1}...w_{n})&:=&
\sum_{\sigma}\e(\sigma)(w_{\sigma(1)}...w_{\sigma(i)} \ , \ w_{\sigma(i+1)}...w_{\sigma(n)}),\\
\unshuff_{i}^{(a,b)}(w_{1}...w_{n})&:=&
\sum_{\sigma}\e(\sigma)(w_{\sigma(1)}...w_{a}...w_{\sigma(i)} \ , \ w_{\sigma(i+1)}...w_{b}...w_{\sigma(n)}),
\end{eqnarray*}
where $\sigma$ is an $(i,n-i)$-unshuffle permutation
and where
$w_{a}\in\{w_{\sigma(1)},...,w_{\sigma(i)}\}$ and $w_{b}\in\{w_{\sigma(i+1)},...,w_{\sigma(n)}\}$.
\end{definition}
For example,
$$
\unshuff_{2}^{(1,3)}(w_{1}w_{2}w_{3}w_{4})=
(w_{1}w_{2},w_{3}w_{4})+(-1)^{(|w_{2}|+|w_{3}|)|w_{4}|}(w_{1}w_{4},w_{2}w_{3}).
$$
The Zinbiel coproduct $\Delta_{z}$ defined in (\ref{zinbcoproduct}) is shortly expressed as follows,
$$
\Delta_{z}(sv_{1}\ot\cdots\ot sv_{n})=
\sum_{i\ge 1}\unshuff_{i}^{(\emptyset,n)}(sv_{1}\ot\cdots\ot sv_{n}).
$$
\begin{proposition}\label{cofreezccoalgebra}
We put $C:=\bar{S}V$.
Then $\bar{T}(sC)$ or $\bar{T}s(\bar{S}V)$
becomes the cofree $\Zinb\odot\s\Com$-coalgebra
of bidegree $(0,1)$.
The commutative coproduct on $\bar{T}sC$ is defined by
\begin{multline*}
\Delta_{c}(sc_{1}\ot\cdots\ot sc_{n})=\\
=\sum_{k\ge 1}^{n-1}\sum_{i\ge 1}^{n-1}
\unshuff^{(k,n)}_{i}(-1)^{|sc_{1}|+\cdots+|sc_{k-1}|}
(sc_{1}\ot\cdots\ot\Delta_{c}(sc_{k})\ot\cdots\ot sc_{n})\\
-\tau\text{\upshape{(1st term)}}
+\Big(\sum_{i\ge 0}^{n-1}\unshuff_{i}(-1)^{|sc_{1}|+\cdots+|sc_{n-1}|}
(sc_{1}\ot\cdots\ot sc_{n-1})\Big)\ot\Delta_{c}(sc_{n}),
\end{multline*}
where $c_{i}\in C=\bar{S}V$ for each $i$,
$\tau$ is the transposition defined in Definition \ref{oddccalg}
and the last tensor product is component wise, i.e.,
if $\Delta_{c}(sc_{n})=\sum sc_{n_{1}}\ot sc_{n_{2}}$,
$$
(x_{1},x_{2})\ot\Delta_{c}(sc_{n})=\sum(-1)^{|x_{2}|(|c_{n_{1}}|+1)}
(x_{1}\ot sc_{n_{1}},x_{2}\ot sc_{n_{2}}).
$$
The Zinbiel coproduct $\Delta_{z}$ is defined by the same manner as (\ref{zinbcoproduct}).
\end{proposition}
\begin{proof}
In general, the dual of shuffle product is an unshuffle coproduct.
The coproduct $\Delta_{c}$ in the proposition is decomposed into
the sum of partial coproducts,
$$
\Delta_{c}=\sum_{i\ge 1}\Delta_{c}^{i},
$$
where $i$ is the length of the word in the left component,
namely, if $\Delta_{c}^{i}(c)=\sum c_{1}\ot c_{2}$, then
the length of $c_{1}$ is $i$.
The dual of the formula proved in Lemma \ref{newlemma1} is $\Delta_{c}^{i}$.
\end{proof}
\begin{example}\label{paLeib}
If $|v_{1}|=|v_{2}|=|v_{3}|=1$, then
\begin{multline*}
\Delta_{c}(sv_{1}v_{2}\ot sv_{3})=\\
=-(sv_{1}\ot sv_{2},sv_{3})+(sv_{2}\ot sv_{1},sv_{3})
+\tau(sv_{1}\ot sv_{2},sv_{3})-\tau(sv_{2}\ot sv_{1},sv_{3})=\\
=-(sv_{1}\ot sv_{2},sv_{3})+(sv_{2}\ot sv_{1},sv_{3})
+(sv_{3},sv_{1}\ot sv_{2})-(sv_{3},sv_{2}\ot sv_{1}),
\end{multline*}
where $\Delta_{c}sv_{3}=0$.
\begin{multline*}
\Delta_{c}(sv_{1}\ot sv_{2}v_{3})
=\big((sv_{1},\emptyset)+(\emptyset,sv_{1})\big)\ot\Delta_{c}(sv_{2}v_{3})=\\
=-(sv_{1}\ot sv_{2},sv_{3})+(sv_{1}\ot sv_{3},sv_{2})+(sv_{3},sv_{1}\ot sv_{2})
-(sv_{2},sv_{1}\ot sv_{3}).
\end{multline*}
In a similar way,
\begin{multline*}
\Delta_{c}(sv_{1}\ot sv_{2}v_{3}\ot sv_{4})=\\
-(sv_{1}\ot sv_{2}\ot sv_{3},sv_{4})+(sv_{1}\ot sv_{3}\ot sv_{2},sv_{4})
-(sv_{2}\ot sv_{3},sv_{1}\ot sv_{4})+(sv_{3}\ot sv_{2},sv_{1}\ot sv_{4})\\
+(sv_{4},sv_{1}\ot sv_{2}\ot sv_{3})-(sv_{4},sv_{1}\ot sv_{3}\ot sv_{2})
+(sv_{1}\ot sv_{4},sv_{2}\ot sv_{3})-(sv_{1}\ot sv_{4},sv_{3}\ot sv_{2},).
\end{multline*}
\end{example}

\subsection{Bar coalgebra}
Let $(\g,(.,.),[.,.])$ be a Lie-Leibniz algebra of bidegree $(0,1)$.
We consider the cofree $\Zinb\odot\s\Com$-coalgebra over $s\g$,
that is, $\bar{T}s\bar{S}s\g$.
\begin{lemma}\label{ctshts}
$\Coder\big(\bar{T}s\bar{S}s\g\big)\cong\Hom\big(\bar{T}s\bar{S}s\g,ss\g\big)$.
\end{lemma}
Let us determine the isomorphism in above lemma.
The space of homomorphisms is decomposed as follows.
\begin{equation*}
\Hom\big(\bar{T}s\bar{S}s\g,ss\g\big)\cong
\bigoplus_{(a_{1},...,a_{n})}
\Big(s\bar{S}^{a_{1}}s\g\ot\cdots\ot s\bar{S}^{a_{n}}s\g\to ss\g\Big).
\end{equation*}
Hence it suffices to consider maps of $s\bar{S}^{a_{1}}s\g\ot\cdots\ot s\bar{S}^{a_{n}}s\g\to ss\g$.
We call the $n$-tuple $(a_{1},...,a_{n})$ an \textbf{arity}
and call the sum $\sum_{i\ge 1}^{n}a_{i}$ the \textbf{total arity}.
If $f$ is a multiplication on $\g$ of arity $(a_{1},...,a_{n})$, i.e.,
$f:\bar{S}^{a_{1}}\g\ot\cdots\ot\bar{S}^{a_{n}}\g\to\g$,
then the coderivation induced from $f$ has the following form,
\begin{equation}\label{defpaf}
\pa_{f}:=ssf\big(s^{-1\ot a_{1}}\ot\cdots\ot s^{-1\ot a_{n}}\big)
(\overbrace{s^{-1}\ot\cdots\ot s^{-1}}^{n}).
\end{equation}
The degree of $\pa_{f}$ is $2+|f|-\sum_{i\ge 1}^{n}a_{i}-n$.\\
\indent
In first we prove Lemma \ref{frelation} below.
\begin{definition}[relation]
Let $(a_{i})$ and $(b_{i})$ be two $n$-tuples of natural numbers.
The two tuples are said to be related at $k$, if
\begin{eqnarray*}
(a_{1},...,a_{k-1})&=&(b_{1},...,b_{k-1}),\\
a_{k}+a_{k+1}&=&b_{k},\\
(a_{k+2},...,a_{n})&=&(b_{k+1},...,b_{n-1}),
\end{eqnarray*}
where $k\in\{1,...,n\}$ and by definition $a_{n+1}\ge 0$.
The relation is denoted by $(a_{i})\sim_{k}(b_{i})$.
\end{definition}
When $k=n$, the defining condition of $(a_{i})\sim_{n}(b_{i})$
is $(a_{1},...,a_{n-1})=(b_{1},...,b_{n-1})$ and $a_{n}\le b_{n}$.
For example $(1,2,3,4)\sim_{2}(1,5,4,1)$
and $(1,2,3,4)\sim_{4}(1,2,3,5)$.
\begin{lemma}\label{frelation}
Suppose that the arity of $f$ is $(a_{1},...,a_{n})$.
If $\pa_{f}$ defined in (\ref{defpaf}) is non trivial on
$s\bar{S}^{b_{1}}s\g\ot\cdots\ot s\bar{S}^{b_{n}}s\g$,
then $(a_{i})\sim_{k}(b_{i})$ for some $k$.
Namely, if the two tuples are non related, then $\pa_{f}$ is trivial automatically.
\end{lemma}
\begin{proof}
The defining relation of coderivations is the dual of the one of derivations.
Therefore, the lemma is followed from Lemma \ref{zcder}.
\end{proof}
We denote the basic element of $\bar{S}^{a_{i}}s\g$
by $c_{i}:=s1_{i}s2_{i}\cdots sa_{i}$ for each $i$.
Here the degrees of labels are zero
$|1|=|2|=\cdots=|a_{i}|=0$ and the one of $c_{i}$ is $|c_{i}|=a_{i}$.
\begin{lemma}
Under the same assumption as above,
\begin{equation*}
\pa_{f}(sc_{1}\ot\cdots\ot sc_{n})=(-1)^{\sum_{i\ge 1}^{n-1}\big((a_{i}+1)(n-i)
+a_{i+1}\sum_{j\ge 1}^{i}a_{j}\big)+\sum_{i\ge 1}^{n}\frac{a_{i}(a_{i}-1)}{2}}
ssf(\bar{c}_{1},...,\bar{c}_{n}),
\end{equation*}
where $\bar{c}_{i}:=1_{i}2_{i}3_{i}\cdots a_{i}\in\bar{S}^{a_{i}}\g$.
\end{lemma}
Suppose that $(a_{i})\sim_{k}(b_{i})$ and take a basic element,
$sc_{1}\ot\cdots\ot sc_{k-1}\ot sc^{\p}_{k}\ot sc_{k+2}\ot\cdots\ot sc_{n+1}
\in s\bar{S}^{b_{1}}s\g\ot\cdots\ot s\bar{S}^{b_{n}}s\g$, where
$$
\Delta_{c}^{a_{k}}sc^{\p}_{k}
=\sum\e(\sigma)(-1)^{|a_{k}|}sc_{k}\ot sc_{k+1},
$$
and where $\Delta_{c}^{a_{k}}$ is the partial coproduct considered
in the proof of Proposition \ref{cofreezccoalgebra}.
The isomorphism in Lemma \ref{ctshts} is completely determined by
\begin{proposition}\label{gformofcoder}
\begin{multline*}
\pa_{f}(sc_{1}\ot\cdots\ot sc_{k-1}\ot sc^{\p}_{k}\ot sc_{k+2}\ot\cdots\ot sc_{n+1})=\\
=\sum\e(\sigma)(-1)^{\sum_{i\ge 1}^{n-1}\big((a_{i}+1)(n-i)
+a_{i+1}\sum_{j\ge 1}^{i}a_{j}\big)+\sum_{i\ge 1}^{n}\frac{a_{i}(a_{i}-1)}{2}+(\sum_{i\ge k+1}^{n}a_{i})+n-k}\\
s\big(sf(\bar{c}_{1},...,\bar{c}_{k},\bar{c}_{k+1},...,\bar{c}_{n})\cdot c_{n+1}\big).
\end{multline*}
\end{proposition}
If $a_{1}=\cdots=a_{n}=1$,
then the sign in above proposition is $\e(\sigma)(-1)^{\frac{n(n-1)}{2}}$.\\
\indent
Now we consider the coderivations induced from the Lie-Leibniz brackets.
Let $(.,.)$ be a Lie bracket on $\g$ of degree $0$
and let $[.,.]$ a Leibniz bracket on $\g$ of degrees $+1$.
The coderivations induced from $(.,.)$ and $[.,.]$ are respectively
\begin{eqnarray*}
\pa_{Lie}&:=&ss(.,.)(s^{-1}\ot s^{-1})s^{-1},\\
\pa_{Leib}&:=&ss[.,.](s^{-1}\ot s^{-1})(s^{-1}\ot s^{-1}).
\end{eqnarray*}
The degrees of $\pa_{Lie}$ and $\pa_{Leib}$ are both $-1$.
We obtain
\begin{proposition}[bar complex]
$\g$ is a Lie-Leibniz algebra of bidegree $(0,1)$ if and only if
$\pa_{LL}:=\pa_{Lie}+\pa_{Leib}$ is a codifferential of degree $-1$
on $\bar{T}s\bar{S}s\g$.
\end{proposition}
As an example of computation, we check that $[\pa_{Lie},\pa_{Leib}]=0$
is equivalent to the distributive law.
In the following we sometimes use the $\mathbb{Z}_{2}$-grading simply,
for instance, $ss=s^{-1}s^{-1}=id$.
Let $s(s1s2)\ot ss3$ be the basic element of $s\bar{S}^{2}s\g\ot ss\g$.
We obtain
\begin{eqnarray}
\label{test1}\pa_{Lie}(s(s1s2)\ot ss3)&=&-ss(1,2)\ot ss3\cong-(1,2)\ot 3,\\
\label{test2}\pa_{Leib}(s(s1s2)\ot ss3)&=&-s(s[1,2]s3)+s(s[2,1]s3).
\end{eqnarray}
The first identity is obvious
and the second one is followed from Proposition \ref{gformofcoder}.
Applying $\pa_{Leib}$ and $\pa_{Lie}$ on (\ref{test1}) and (\ref{test2}),
\begin{eqnarray*}
\pa_{Leib}\pa_{Lie}(s(s1s2)\ot ss3)&=&[(1,2),3],\\
\pa_{Lie}\pa_{Leib}(s(s1s2)\ot ss3)&=&-([1,2],3)+([2,1],3),
\end{eqnarray*}
which yields
$$
[\pa_{Leib},\pa_{Lie}](s(s1s2)\ot ss3)=[(1,2),3]-([1,2],3)+([2,1],3).
$$
In a similar way, one can prove
$$
[\pa_{Leib},\pa_{Lie}](ss1\ot s(s2s3))=[1,(2,3)]-([1,2],3)-(2,[1,3]),
$$
where $ss1\ot s(s2s3)$ is the basic element of $ss\g\ot s\bar{S}^{2}s\g$.
\begin{remark}[Koszul dual cooperad]
The Koszul dual cooperad of $\LL$, which is usually denoted by $\LL^{\bot}$,
is defined as a functor such that
$$
\LL^{\bot}(\g)=s^{-2}\big(\bar{T}s\bar{S}s\g\big).
$$
This coalgebra $\LL^{\bot}(\g)$ has the universal allow
$\LL^{\bot}(\g)\to\g$.
\end{remark}

\section{$\LL_{\infty}$-algebras}

Since the operad $\LL$ is Koszul, the notion of
$\LL_{\infty}$-operad is well-defined
as the quasi-free resolution over $\LL$ (\cite{GK}).
In this section we will study two subclasses of $\LL_{\infty}$-algebras.
\begin{definition}
A graded space $\g$ is called an $\LL_{\infty}$-algebra or strong homotopy Lie-Leibniz algebra,
if there exists a codifferential, $\pa_{shLL}$, of degree $-1$ on $\bar{T}s\bar{S}s\g$.
\end{definition}
The codifferential $\pa_{shLL}$ is expressed as a sum of coderivations, $\pa_{(a_{1},...,a_{n})}$, of arity $(a_{1},...,a_{n})$.
The coderivation $\pa_{(a_{1},...,a_{n})}$ is induced from a higher homotopy defined on $\g$,
$$
l^{(a_{1},...,a_{n})}:\bar{S}^{a_{1}}\g\ot\cdots\ot\bar{S}^{a_{n}}\g\to \g.
$$
The degree of $l_{n}^{(a_{1},...,a_{n})}$ is $n+\sum_{i}a_{i}-3$ (recall (\ref{defpaf}).)
\subsection{Sh Leibniz algebra with invariant 2-forms}

In first we recall the notion of strong homotopy (sh, for short) Leibniz algebra (\cite{AP}).
\begin{definition}
Let $\g$ be a graded space and let $D_{1},D_{2},...$
be a system of coderivations defined on the Zinbiel coalgebra $\bar{T}\g$.
Here $D_{n}:\g^{\ot n}\to\g$ and $|D_{n}|:=2n-3$ for each $n$.
If $D_{shLeib}:=\sum D_{i}$ is an inhomogeneous codifferential on $\bar{T}\g$,
then $(\g,D_{1},D_{2},...)$ is called an \textbf{odd} sh Leibniz algebra.
\end{definition}
\indent
The coderivation associated with $D_{n}$ is defined as follows.
$$
\ti{\pa}_{n}:=ssD_{n}(s^{-1})^{\ot n}(s^{-1})^{\ot n}.
$$
The degree of $\ti{\pa}_{n}$ is $-1$, because $|D_{n}|=2n-3$.
In the category of $\mathbb{Z}_{2}$-graded spaces,
$\ti{\pa}_{n}\cong (-1)^{\frac{n(n-1)}{2}}D_{n}$.
Sh Leibniz algebras are special examples of sh Lie-Leibniz algebras, that is,
$D_{shLeib}D_{shLeib}=0$ if and only if $\pa_{shLeib}\pa_{shLeib}=0$, where
$$
\pa_{shLeib}:=\ti{\pa}_{1}+\ti{\pa}_{2}+\cdots
$$
which lives in $\bar{T}s\bar{S}s\g$.\\
\indent
We introduce a new type of strong homotopy algebra.
\begin{definition}
Let $(\g,\pa_{shLeib},(.,.))$ be an odd sh Leibniz algebra
equipped with a Lie bracket of degree $0$.
The space $\g$ is called an \textbf{invariant sh Leibniz algebra},
if $\pa_{Lie}+\pa_{shLeib}$ is a codifferential on $\bar{T}s\bar{S}s\g$.
\end{definition}
\begin{lemma}\label{cocyclecondition}
$[\pa_{Lie},\ti{\pa}_{n}]=0$ if and only if
$$
D_{n}\big(1,...,(k,k+1),...,n+1\big)=
\big(D_{n}(1,...,n),n+1\big)-\big(D_{n}(1,...,k+1,k,...,n),n+1\big),
$$
for each $1\le k\le n$.
\end{lemma}
\begin{proof}
From Proposition \ref{gformofcoder}, we obtain
\begin{multline*}
\ti{\pa}_{n}\big(ss1\ot\cdots\ot s\big(sk\cdot s(k+1)\big)\ot\cdots \ot ss(n+1)\big)=\\
=(-1)^{\frac{n(n-1)}{2}}s\big(sD_{n}(1,...,n)\cdot s(n+1)-sD_{n}(1,...,k+1,k,...,n)\cdot s(n+1)\big).
\end{multline*}
Therefore, up to $(-1)^{\frac{n(n-1)}{2}}$ and up to $ss=id$,
\begin{eqnarray*}
\pa_{Lie}\ti{\pa}_{n}&=&(D_{n}(1,...,n),n+1)-(D_{n}\tau^{k,k+1}(1,...,n),n+1),\\
\ti{\pa}_{n}\pa_{Lie}&=&-D_{n}(1,...,(k,k+1),...,n),
\end{eqnarray*}
which yields the identity of the lemma.
\end{proof}
\begin{corollary}[Cartan type 3-forms]
If $\g$ is a Lie-Leibniz algebra of bidegree $(0,1)$,
then a 3-form $\psi:=([.,.],.)$ of arity $(1,1,1)$ is a cocycle
with respect to $\partial_{LL}:=\pa_{Lie}+\pa_{Leib}$.
\end{corollary}
\begin{proof}
By a direct computation.
\end{proof}
\begin{theorem}
Let $(\g,d_{0},(.,.))$ be a dg Lie algebra of degree $0$,
where the degree of $d_{0}$ is $-1$.
Suppose that $d_{0}+\hbar d_{1}+\hbar^{2}d_{2}+\cdots$
is an inhomogeneous deformation of $d_{0}$.
Here the degree of $d_{n}$ is $2n-1$ for each $n$.
Define a system of \textbf{higher derived brackets}
\begin{eqnarray*}
D_{0}&:=&(.,.)\\
D_{1}&:=&d_{0}\\
D_{2}&:=&(d_{1}.,.)\\
D_{3}&:=&((d_{2}.,.),.)\\
\cdots&\cdots&\cdots.
\end{eqnarray*}
Then $(\g,D_{0},D_{1},...)$ becomes an invariant sh Leibniz algebra.
\end{theorem}
\begin{proof}
It is known that the sum of derivations $D_{shLeib}:=D_{1}+D_{2}+\cdots$ becomes an
odd sh Leibniz structure (\cite{U1}).
Hence it suffices to show that
$D_{shLeib}$ is a cocycle of $D_{0}(\cong\pa_{Lie})$.
From the Jacobi identity of the Lie bracket,
the higher derived bracket satisfies
\begin{multline*}
((((((d_{n-1}1,2),3),...,k-1),(k,k+1),),k+2)...,n+1)=\\
=(((((d_{n-1}1,2),3),...,k),k+1),...,n+1)-(((((d_{n-1}1,2),3),...,k+1),k),...n+1),
\end{multline*}
which is just the identity proved in Lemma \ref{cocyclecondition}.
\end{proof}
\subsection{Derived homotopies}

In this section we introduce a new type of derived bracket construction,
i.e., the derived homotopy construction on sh Lie algebras
and we prove that a system of derived homotopies
defines an sh Lie-Leibniz algebra.\\
\indent
In first we recall the notion of sh Lie algebra.
\begin{definition}
A graded space $\g$ is an sh Lie algebra
if and only if there exists a codifferential $\pa_{shLie}$ of degree $-1$ on
the commutative coalgebra $\bar{S}(s\g)$.
\end{definition}
It is obvious that sh Lie algebras are also special $\LL_{\infty}$-algebras.
We put $\pa_{shLie}=\pa_{1}+\pa_{2}+\cdots$, where $\pa_{n}:\bar{S}^{n}(s\g)\to s\g$.
The coderivation $\pa_{n}$ is induced from
a higher homotopy $l_{n}:\bar{S}^{n}\g\to\g$,
$$
\pa_{n}=sl_{n}(\overbrace{s^{-1}\ot\cdots\ot s^{-1}}^{n}).
$$
The degree of $l_{n}$ is $n-2$, in particular,
the degree of $l_{1}$ (differential) is $-1$.
\medskip\\
\indent
Suppose that $(\g,d,l_{1},l_{2},...)$ is an sh Lie algebra
equipped with a differential $d$ of degree $+1$
and that $d$ is a derivation of $l_{n}$ for each $n$.
Define a collection of \textbf{derived homotopies}:
\begin{eqnarray*}
&&l_{1}\\
&&l_{2} \ , \ l_{2}(d\ot 1)\\
&&l_{3} \ , \ -l_{3}(d\ot 1\ot 1) \ , \ -l_{3}(d\ot d\ot 1)\\
&&\cdots\\
&&l_{m} \ , \ \cdots \ , \ (-1)^{\frac{(2m+i-1)i}{2}}l_{m}(d^{\ot i},1^{\ot m-i}) \ , \ \cdots \ , \
(-1)^{\frac{(3m-2)(m-1)}{2}}l_{m}(d^{\ot m-1},1)\\
&&\cdots.
\end{eqnarray*}
The arity of $(\pm)^{i}_{m}l_{m}(d^{\ot i}\ot 1^{\ot m-i})$
is by definition $(\overbrace{1,...,1}^{i},m-i)$, where
$$
(\pm)^{i}_{m}:=(-1)^{\frac{(2m+i-1)i}{2}}=(-1)^{m+(m+1)+\cdots+(m+i-1)},
$$
which is the sign associated with the bracket degree (see Remark \ref{bracketdegree}.)
This sign is hierarchic, for instance,
$-l_{3}(d\ot 1\ot 1)$ is the derived homotopy of $l_{3}$
and $-l_{3}(d\ot d\ot 1)$ is the derived homotopy of $-l_{3}(d\ot 1\ot 1)$.\\
\indent
The main theorem of this section is as follows.
\begin{theorem}\label{lasttheorem}
The system of derived homotopies above
defines an $\LL_{\infty}$-algebra structure on $\g$.
\end{theorem}
The coderivation induced from $(\pm)^{i}_{m}l_{m}(d^{\ot i}\ot 1^{\ot n-i})$ is given by
$$
\pa_{m}^{(i)}:=(\pm)^{i}_{m}
s\Big(sl_{m}(d^{\ot i}\ot 1^{\ot m-i})(s^{-1})^{\ot m}\Big)
\big((s^{-1})^{\ot i}\ot s^{-1}\big),
$$
whose degree is $-1$.
If $i=0$, then $\pa^{(0)}_{m}=s\pa_{m}s^{-1}$.
In the category of $\mathbb{Z}_{2}$-graded spaces,
$$
\pa_{m}^{(i)}=(\pm)^{i}_{m}(-1)^{\frac{i(i-1)}{2}+i(m-i)}l_{m}(d^{\ot i},1^{\ot m-i})
\big(1^{\ot i}\ot(s^{-1})^{\ot m-i}s^{-1}\big).
$$
For the basic element of $(ss\g)^{\ot i}\ot s\bar{S}^{m-i}s\g$,
\begin{multline}\label{theX}
\pa^{(i)}_{m}\big(ss1\ot\cdots\ot ssi\ot s(s(i+1)\cdots s(n))\big)\\
=(\pm)^{i}_{m}(-1)^{\frac{i(i-1)}{2}+i(m-i)+\frac{(m-i)(m-i-1)}{2}}l_{m}(d1,...,di,i+1,...,m)\\
=(\pm)^{i}_{m}(-1)^{\frac{m(m-1)}{2}}l_{m}(d1,...,di,i+1,...,m)\\
=(\pm)^{i}_{m}\{d1,...,di,i+1,...,m\},
\end{multline}
where $\{,...,\}:=(-1)^{\frac{m(m-1)}{2}}l_{m}$.
The coderivation induced from the derived homotopies is
$$
\pa_{shLL}:=\sum_{m\ge 1}\sum_{i\ge 0}^{m-1}\pa_{m}^{(i)}.
$$
We should prove $[\pa_{shLL},\pa_{shLL}]=0$, or equivalently,
\begin{equation}\label{shllshll}
\sum_{i+j=Const\atop m+n=Const}[\pa_{m}^{(i)},\pa_{n}^{(j)}]=0.
\end{equation}
\begin{lemma}
Let $x$ be an element of
$s\bar{S}^{a_{1}}s\g\ot\cdots\ot s\bar{S}^{a_{j+1}}s\g$.
\begin{eqnarray*}
\pa_{n}^{(j)}(x)\neq 0\Rightarrow
(a_{1},...,a_{j+1})=
\left\{
\begin{array}{l}
(\overbrace{1,...,1}^{j-1},n-j+1,a_{j+1}),\\
(\overbrace{1,...,1}^{j},a_{j+1}\ge n-j),\\
\end{array}
\right.
\end{eqnarray*}
and if $(a_{1},...,a_{j+1})=(\overbrace{1,...,1,2,1,...,1}^{j-1},n-j,a_{j+1})$,
then $\pa_{n}^{(j)}(x)=0$.
\end{lemma}
\begin{proof}
The arity of $\pa_{n}^{(j)}(x)$ is $(\overbrace{1,...,1}^{j},n-j)$.
The first part of the lemma is followed from Lemma \ref{frelation}.
The second part is from the anti-symmetry of $l_{n}$, for example,
up to sign,
\begin{eqnarray*}
\pa_{3}^{(2)}(s(s1s2)\ot 3\ot 4)&\sim& l_{3}(d1,d2,3)\cdot 4-l_{3}(d2,d1,3)\cdot 4\\
&=&l_{3}(d1,d2,3)\cdot 4-l_{3}(d1,d2,3)\cdot 4=0.
\end{eqnarray*}
\end{proof}
\begin{lemma}\label{xlemma}
The arity of $\sum_{i+j=Const\atop m+n=Const}[\pa_{m}^{(i)},\pa_{n}^{(j)}]$
is $(\overbrace{1,...,1}^{i+j},m+n-1-i-j)$.
\end{lemma}
\begin{proof}
The total arity of $[\pa_{m}^{(i)},\pa_{n}^{(j)}]$ is $m+n-1$.
So we take an element $x$ in $sS^{a_{1}}s\g\ot\cdots\ot sS^{a_{i+j+1}}s\g$,
where $a_{1}+\cdots+a_{i+j+1}=m+n-1$.
\begin{claim}
\begin{eqnarray*}\label{last01}
\pa_{m}^{(i)}\pa_{n}^{(j)}(x)\neq 0\Rightarrow
(a_{1},...,a_{i+j+1})=
\left\{
\begin{array}{ll}
(A1) &(1,...,1,n-j,\overbrace{1,...,1}^{\ge 0},m-i)\\
(A2) & (1,...,1,n-j+1,\overbrace{1,...,1}^{\ge 0},m-i-1)\\
(B) & (1,...,1,m+n-1-i-j),\\
\end{array}
\right.
\end{eqnarray*}
because $\pa_{n}^{(j)}(x)\in(ss\g)^{\ot i}\ot s(S^{m-i}s\g)$.
\end{claim}
The lemma says that in the cases of (A1) and (A2), $[\pa_{shLL},\pa_{shLL}](x)=0$.
Suppose (A1), i.e.,
$(a_{1},...,a_{i+j+1}):=(\overbrace{1,...,1}^{p},q,\overbrace{1,...,1}^{i+j-p-1},r)$.
Here $q+r=m+n-i-j$, because $a_{1}+\cdots+a_{i+j+1}=m+n-1$.
Then we obtain
$$
[\pa_{shLL},\pa_{shLL}](x)=
\sum_{n-j=q \atop j\le p}
(\pa_{m}^{(i)}\pa_{n}^{(j)}+\pa_{m}^{(i-1)}\pa_{n}^{(j+1)})(x),
$$
where $m+n=Const$ and $i+j=Const$.
It is obvious that the case of (A2) is included in this case.
For example, if $(a_{1},...,a_{i+j+1})=(1,1,3,1,2)$, then $i+j=4$, $m+n=9$, $n-j=3$, $p=2$ and
$$
[\pa_{shLL},\pa_{shLL}](x)=\Big(
(\pa^{(2)}_{4}\pa^{(2)}_{5}+\pa^{(1)}_{4}\pa^{(3)}_{5})+
(\pa^{(3)}_{5}\pa^{(1)}_{4}+\pa^{(2)}_{5}\pa^{(2)}_{4})+
(\pa^{(4)}_{6}\pa^{(0)}_{3}+\pa^{(3)}_{6}\pa^{(1)}_{3})\Big)(x).
$$
By a direct computation, one can prove that
\begin{equation}\label{minj=0}
(\pa_{m}^{(i)}\pa_{n}^{(j)}+\pa_{m}^{(i-1)}\pa_{m}^{(j+1)})(x)=0,
\end{equation}
which yields the lemma.
In Appendix we give a proof of (\ref{minj=0}).
\end{proof}
\begin{lemma}\label{thelastlemma}
\begin{multline*}
\sum_{i+j=k}\pa_{m}^{(i)}\pa_{n}^{(j)}
\Big(ss1\ot\cdots\ot ssk\ot s\big(s(k+1)\cdots s(m+n-1)\big)\Big)=\\
=(-1)^{k}\pa_{m}\pa_{n}
\big(sd1\cdots sdk\cdot s(k+1)\cdots s(m+n-1)\big),
\end{multline*}
where $sd1\cdots sdk\cdot s(k+1)\cdots s(m+n-1)\in\bar{S}^{m+n-1}(s\g)$
and $\pa_{m},\pa_{n}$ are the coderivations of the sh Lie algebra.
\end{lemma}
\begin{proof}
See Appendix.
\end{proof}
The theorem, i.e., equation (\ref{shllshll})
is followed from Lemmas \ref{xlemma} and \ref{thelastlemma}.
\begin{corollary}
The diagonal part $l_{1},l_{2}(d\ot 1),-l_{3}(d\ot d\ot 1),...$ defines an sh Leibniz algebra of odd.
\end{corollary}

\section{Appendix}

\noindent
\textbf{A0. Manin white product} (cf. Vallette \cite{V}).
Given two bq operads $\P_{i}=(E_{i},R_{i})$, $i\in\{1,2\}$,
the Manin white product is defined by
$$
\P_{1}\c_{M}\P_{2}:=(E_{1}\ot E_{2},R_{12}),
$$
where $R_{12}:=\Psi^{-1}\big(R_{1}\ot\T E_{2}(3)+\T E_{1}(3)\ot R_{2}\big)$
and where $\Psi$ is the lift of the natural map $E_{1}\ot E_{2}\to\T E_{1}\ot\T E_{2}$,
$$
\begin{CD}
E_{1}\ot E_{2}@>\subset>>\T(E_{1}\ot E_{2})\\
@| @VV{\Psi}V \\
E_{1}\ot E_{2}@>>>\T E_{1}\ot\T E_{2}.
\end{CD}
$$
\medskip\\
\noindent
\textbf{A1. Proof of (\ref{minj=0})}.
In the following we put $(1,...,m):=(-1)^{\frac{m(m-1)}{2}}l_{m}(1,...,m)$ for each $m$.
Take the basic element in
$\g^{\ot p}\ot s(\bar{S}^{n-j}s\g)\ot\g^{\ot i+j-p-1}\ot s(\bar{S}^{m-i}s\g)$,
\begin{multline}\label{basicelement}
1\ot\cdots\ot p\ot s\big(s(p+1)\cdots s(p+n-j)\big)\ot
(p+n-j+1)\ot\cdots\ot(n+i-1)\ot \\s\big(s(n+i)\cdots s(m+n-1)\big),
\end{multline}
where $ss=id$.
Apply $\pa_{n}^{(j)}$ to (\ref{basicelement}).
\begin{multline}\label{panj}
\pa_{n}^{(j)}(\ref{basicelement})=\sum_{\sigma}
\sigma_{1}\ot\cdots\ot\sigma_{p-j}\ot\pa_{n}^{(j)}\Big(\sigma_{p-j+1}\ot\cdots\ot\sigma_{p}\ot
s\big(s(p+1)\cdots s(p+n-j)\big)\Big)\\
\ot(p+n-j+1)\ot\cdots\ot(n+i-1)\ot s\big(s(n+i)\cdots s(m+n-1)\big),
\end{multline}
where $\sigma$ is a $(p-j,j)$-unshuffle permutation.
By (\ref{theX}),
\begin{multline*}
\pa_{n}^{(j)}\Big(\sigma_{p-j+1}\ot\cdots\ot\sigma_{p}\ot
s\big(s(p+1)\cdots s(p+n-j)\big)\Big)=\\
=(\pm)_{n}^{(j)}\{d\sigma_{p-j+1},...,d\sigma_{p},p+1,...,p+n-j\}.
\end{multline*}
Hence (\ref{panj}) is equal to
\begin{multline*}
\pa_{n}^{(j)}(\ref{basicelement})=(\pm)_{n}^{(j)}\sum_{\sigma}
\sigma_{1}\ot\cdots\ot\sigma_{p-j}\ot
\{d\sigma_{p-j+1},...,d\sigma_{p},p+1,...,p+n-j\}\\
\ot(p+n-j+1)\ot\cdots\ot(n+i-1)\ot s\big(s(n+i)\cdots s(m+n-1)\big),
\end{multline*}
Apply $\pa_{m}^{(i)}$ to above.
\begin{multline}\label{pamipanj}
\pa_{m}^{(i)}\pa_{n}^{(j)}(\ref{basicelement})=(\pm)_{m}^{i}(\pm)_{n}^{j}(\pm)_{1}\sum_{\sigma}
\big\{d\sigma_{1},...,d\sigma_{p-j},
d\{d\sigma_{p-j+1},...,d\sigma_{p},p+1,...,p+n-j\},\\
\overbrace{d(p+n-j+1),...,d(n+i-1)}^{i+j-p-1},n+i,...,m+n-1\big\},
\end{multline}
where
$(\pm)_{1}=(-1)^{(n+j)(i+j-p-1)+(n+j)(m-i+1)}$.
By using the anti-symmetry of $l_{m}$, remove $d\{d\sigma_{p-j+1},...,d\sigma_{p},p+1,...,p+n-j\}$
to the left-position of $n+i$,
\begin{multline}\label{pamipanj2}
(\ref{pamipanj})=(\pm)_{m}^{i}(\pm)_{n}^{j}(\pm)_{1}(\pm)_{2}\sum_{\sigma}
\big\{d\sigma_{1},...,d\sigma_{p-j},d(p+n-j+1),...,d(n+i-1),\\
d\{d\sigma_{p-j+1},...,d\sigma_{p},p+1,...,p+n-j\},n+i,...,m+n-1\big\},
\end{multline}
where $(\pm)_{2}=(-1)^{(n+j+1)(i+j-p-1)+(i+j-p-1)}$.
The sign of (\ref{pamipanj2}) is
\begin{equation}\label{theY}
(\pm)_{m}^{i}(\pm)_{n}^{j}(\pm)_{1}(\pm)_{2}
=(\pm)_{m}^{i}(\pm)_{n}^{j}(-1)^{(n+j)(m-i+1)}=(\pm)_{m}^{i}(\pm)_{n}^{j+1}(-1)^{(n+j)(m-i)},
\end{equation}
where $(\pm)_{n}^{j}(-1)^{n+j}=(\pm)_{n}^{j+1}$ is used.\\
\indent
We compute $\pa_{m}^{(i-1)}\pa_{n}^{(j+1)}(\ref{basicelement})$.
Up to $\pa_{m}^{(i-1)}(-)$,
\begin{multline}\label{42}
\pa_{n}^{(j+1)}(\ref{basicelement})
\equiv\sum_{\sigma}\sigma_{1}\ot\cdots\ot\sigma_{p-j}\ot(p+n-j+1)\ot\cdots\ot(n+i-1)\ot\\
\pa_{n}^{(j+1)}\Big(\overbrace{\sigma_{p-j+1}\ot\cdots\ot\sigma_{p}}^{j}\ot
s\big(\overbrace{s(p+1)\cdots s(p+n-j)}^{n-j}\big)\ot s\big(s(n+i)\cdots s(m+n-1)\big)\Big).
\end{multline}
The last component of (\ref{42}) is from Proposition \ref{gformofcoder},
\begin{multline*}
\pa_{n}^{(j+1)}\Big(\sigma_{p-j+1}\ot\cdots\ot\sigma_{p}\ot
s\big(s(p+1)\cdots s(p+n-j)\big)\ot s(s(n+i)\cdots s(m+n-1)\big)\Big)\\
=s\sum_{k}(\pm)_{n}^{(j+1)}(\pm)s\{d\sigma_{p-j+1},...,d\sigma_{p},d(p+k),
p+1,...,\hat{(p+k)},...,p+n-j\}s(n+i)\cdots s(m+n-1)\\
=s\sum_{k}(\pm)_{n}^{(j+1)}(\pm)s\{d\sigma_{p-j+1},...,d\sigma_{p},
p+1,...,d(p+k),...,p+n-j\}\cdots s(m+n-1)\\
=(\pm)_{n}^{(j+1)}s\Big(
sd\{d\sigma_{p-j+1},...,d\sigma_{p},p+1,...,p+n-j\}\cdots s(m+n-1)
\Big).
\end{multline*}
Hence
\begin{multline}\label{43}
\pa_{n}^{(j+1)}(\ref{basicelement})
\equiv(\pm)_{n}^{(j+1)}\sum_{\sigma}\sigma_{1}\ot\cdots\ot\sigma_{p-j}\ot(p+n-j+1)\ot\cdots\ot(n+i-1)\ot\\
s\big(
sd\{d\sigma_{p-j+1},...,d\sigma_{p},p+1,...,p+n-j\}\cdots s(m+n-1)
\big).
\end{multline}
Apply $\pa_{m}^{(i-1)}$ to (\ref{43}).
Then we get (\ref{pamipanj2}) up tp $-1$-times.
\begin{multline}\label{44}
\pa_{m}^{(i-1)}\pa_{n}^{(j+1)}(\ref{basicelement})=(\pm)^{(i-1)}_{m}(\pm)^{(j+1)}_{n}(\pm)_{3}
\sum_{\sigma}\big\{d\sigma_{1},...,d\sigma_{p-j},d(p+n-j+1),...,d(n+i-1),\\
d\{d\sigma_{p-j+1},...,d\sigma_{p},p+1,...,p+n-j\},...,m+n-1\big\}
\end{multline}
Here $(\pm)_{3}=(-1)^{(m-i)(n+j-1)}$ and the sign of (\ref{44}) is
\begin{multline}\label{theZ}
(\pm)^{i-1}_{m}(\pm)^{j+1}_{n}(\pm)_{3}=
(\pm)^{i-1}_{m}(\pm)^{j+1}_{n}(-1)^{(m-i)(n+j-1)}=\\
-(\pm)^{i-1}_{m}(\pm)^{j+1}_{n}(-1)^{m+i-1}(-1)^{(m-i)(n+j)}=
-(\pm)^{i}_{m}(\pm)^{j+1}_{n}(-1)^{(m-i)(n+j)},
\end{multline}
where $(\pm)^{i-1}_{m}(-1)^{m+i-1}=(\pm)_{m}^{i}$.
We get $(\ref{theZ})=-(\ref{theY})$.
\medskip\\
\noindent
\textbf{A2. Proof of Lemma \ref{thelastlemma}}.
Let us compute
\begin{equation}\label{lp1}
\pa_{n}^{(j)}
\Big(1\ot\cdots\ot k\ot s\big(s(k+1)\cdots s(m+n-1)\big)\Big),
\end{equation}
where $ss=id$ and $k=i+j$.
\begin{multline}\label{lp2}
(\ref{lp1})=
\sum_{\sigma}\sigma_{1}\ot\cdots\ot\sigma_{i}\ot
\pa_{n}^{j}\Big(\sigma_{i+1}\ot\cdots\ot\sigma_{k}\ot s\big(s(k+1)\cdots s(m+n-1)\big)\Big)=\\
=\sum_{\sigma,\eta}\e(\eta)(\pm)^{j}_{n}\sigma_{1}\ot\cdots\ot\sigma_{i}\ot
s\big(
s\{d\sigma_{i+1},...,d\sigma_{k},\eta_{k+1},...,\eta_{n+i}\}s\eta_{n+i+1}\cdots s\eta_{m+n-1}
\big).
\end{multline}
Apply $\pa^{(i)}_{m}$ to above.
\begin{multline}\label{lp3}
\pa^{(i)}_{m}(\ref{lp2})=\sum_{\sigma,\eta}\e(\eta)(\pm)^{i}_{m}(\pm)^{j}_{n}(-1)^{(n+j)(m-i-1)}\\
\big\{d\sigma_{1},...,d\sigma_{i},\{d\sigma_{i+1},...,d\sigma_{i+j},\eta_{i+j+1},...,\eta_{n+i}\},
\eta_{n+i+1},...,\eta_{m+n-1}\big\}.
\end{multline}
We should compute
\begin{equation}\label{lp4}
\pa_{n}\big(sd1\cdots sdk\cdot s(k+1)\cdots s(m+n-1)\big).
\end{equation}
In (\ref{lp4}), there exists the following term,
\begin{equation}\label{lp5}
\sum_{\sigma,\eta}\e(\eta)
sd\sigma_{1}\cdots sd\sigma_{i}\cdot\pa_{n}(sd\sigma_{i+1}\cdots sd\sigma_{i+j}\cdot
\overbrace{s\eta_{i+j+1}\cdots s\eta_{n+i}}^{n-j})\cdot
\overbrace{s\eta_{n+i+1}\cdots s\eta_{m+n-1}}^{m-i-1},
\end{equation}
where $\e(\sigma)=1$, because $sd$ is even.
\begin{multline}\label{lp6}
(\ref{lp5})=\sum_{\sigma,\eta}\e(\eta)(-1)^{\frac{(n-j)(n-j-1)}{2}}\\
sd\sigma_{1}\cdots sd\sigma_{i}\cdot sl_{n}(d\sigma_{i+1},...,d\sigma_{i+j},\eta_{i+j+1},...,\eta_{n+i})\cdot s\eta_{n+i+1}\cdots s\eta_{m+n-1}\\
=\sum_{\sigma,\eta}\e(\eta)(-1)^{\frac{(n-j)(n-j-1)}{2}+\frac{n(n-1)}{2}}\\
sd\sigma_{1}\cdots sd\sigma_{i}\cdot s\{d\sigma_{i+1},...,d\sigma_{i+j},\eta_{i+j+1},...,\eta_{n+i}\}\cdot s\eta_{n+i+1}\cdots s\eta_{m+n-1}.
\end{multline}
Apply $\pa_{m}$ to (\ref{lp6}).
\begin{multline}\label{lp7}
\pa_{m}(\ref{lp6})=\sum_{\sigma,\eta}\e(\eta)(-1)^{\frac{(n-j)(n-j-1)}{2}+\frac{n(n-1)}{2}+(1+n+j)(m-i-1)
+\frac{(m-i-1)(m-1-2)}{2}+\frac{m(m-1)}{2}}\\
\big\{d\sigma_{1},...,d\sigma_{i},\{d\sigma_{i+1},...,d\sigma_{i+j},\eta_{i+j+1},...,\eta_{n+i}\},\eta_{n+i+1},...,\eta_{m+n-1}\big\}.
\end{multline}
One can prove that the sign of (\ref{lp3}) is equal to the one of (\ref{lp7})
up to $(-1)^{k}$.

\end{document}